\documentclass[reqno,11pt]{amsart}
\usepackage{amsfonts,amsmath,amsthm,amssymb,stmaryrd}
\usepackage{graphicx}
\usepackage{epstopdf}
\usepackage{esint}
\usepackage{color}
\usepackage[RGB]{xcolor}
\newtheorem{Lemma}{Lemma}[section]
\newtheorem{Theorem}{Theorem}[section]
\newtheorem{Definition}{Definition}[section]
\newtheorem{Proposition}{Proposition}[section]
\newtheorem{Remark}{Remark}[section]
\newtheorem{Corollary}{Corollary}[section]

\numberwithin{equation}{section} 

\newcommand{\LV}{\left|}
\newcommand{\RV}{\right|}
\newcommand{\LN}{\left\|}
\newcommand{\RN}{\right\|}
\newcommand{\LB}{\left[}
\newcommand{\RB}{\right]}
\newcommand{\LC}{\left(}
\newcommand{\RC}{\right)}

\newcommand{\LCB}{\left\{}
\newcommand{\RCB}{\right\}}

\newcommand{\R}{\mathbb{R}} 

\newcommand{\beq}{\begin{equation}}
\newcommand{\eeq}{\end{equation}}
\newcommand{\bes}{\begin{eqnarray}}
\newcommand{\ees}{\end{eqnarray}}
\def\bega{\begin{array}}
\def\enda{\end{array}}
\def\begi{\begin{itemize}}
\def\endi{\end{itemize}}
\def\bel{\begin{equation}\label}

\usepackage[left=1 in, right=1 in,top=1 in, bottom=1 in]{geometry}

\title[Lipschitz metric for the Novikov equation]{Lipschitz metric for the Novikov equation}
\author[H. Cai]{Hong Cai} 
\address{Hong Cai\newline
School of Mathematical Sciences, Xiamen University, Fujian, Xiamen, 361005, China, and School of Mathematics, Georgia Institute of Technology, Atlanta, GA 30332, USA.}
\email{caihong19890418@163.com}

\author[G. Chen]{Geng Chen} 
\address{Geng Chen \newline
Department of Mathematics, University of Kansas, Lawrence, KS 66045}
\email{gengchen@ku.edu}

\author[R. M. Chen]{Robin Ming Chen} 
\address{Robin Ming Chen \newline
Department of Mathematics, University of Pittsburgh, Pittsburgh PA 15260}
\email{mingchen@pitt.edu}

\author[Y. Shen]{Yannan Shen}
\address{Yannan Shen \newline
Department of Mathematics, California Sate University, 
Northridge, CA 91330}
\email{yannan.shen@csun.edu}

\begin{document}

\maketitle

\begin{abstract}
We consider the Lipschitz continuous dependence of solutions for the Novikov equation with respect to the initial data. In particular, we construct a Finsler type optimal transport metric which renders the solution map Lipschitz continuous on bounded set of $H^1(\R)\cap W^{1,4}(\R)$, although it is not Lipschitz continuous under the natural Sobolev metric from energy law due to the finite time gradient blowup. By an application of Thom's transversality Theorem, we also prove that when the initial data are in an open dense set of $H^1(\R)\cap W^{1,4}(\R)$, the solution is piecewise smooth. This generic regularity result helps us extend the Lipschitz continuous metric to the general weak solutions.
\end{abstract}

\section{Introduction}
Many evolutionary partial differential equations (PDEs) have the general form
\begin{equation*}
u_t + Lu = 0, \quad u(0) = u_0.
\end{equation*}
Classical well-posedness theory suggests the existence of a continuous semigroup of solutions, at least for a short time. For a large class of semi-linear PDEs, basic techniques such as Picard and Duhamel iteration can be applied to obtain Lipschitz continuity of the semigroup, that is, for any pair of solutions $u, v$, it holds
\begin{equation}\label{Lip_semigroup}
{d \over dt} \| u(t) - v(t) \| \lesssim \|u(t) - v(t) \|
\end{equation}
for a suitable (Sobolev) norm. Typical examples include the Korteweg-de Vries (KdV) equation, the nonlinear Schr\"odinger (NLS) equation, the semi-linear wave equation, and so on. 

On the other hand, a noteworthy exception is provided by many {\it quasi-linear} equations. Due to the dominating nonlinearity, the initial information of the data can determine the later dynamics in a substantial way. In particular, solutions with smooth initial data can lose regularity in finite time, and one cannot in general expect \eqref{Lip_semigroup} to hold true under the natural energy norms.

In this paper, we would like to address the issue of the Lipschitz continuity of the flow map using the following quasi-linear equation, namely the Novikov equation
\begin{equation}\label{Novikov}
 u_t-u_{xxt}+4u^2u_x=3uu_xu_{xx}+u^2u_{xxx}.
\end{equation}
This equation was derived by Novikov \cite{Nov} in a symmetry classification of nonlocal PDEs with cubic nonlinearity, and can in some sense be related to the well-known Camassa-Holm equation \cite{ch,ff}. In fact writing the Novikov equation \eqref{Novikov} in the weak form
\begin{equation}\label{weakNov}
u_t + u^2u_x + \partial_x(1 - \partial_x^2)^{-1} \LC u^3 + {3\over2} uu_x^2 \RC + (1 - \partial_x^2)^{-1} \LC {1\over2} u_x^3 \RC = 0,
\end{equation}
one may recognize the similarities with the Camassa-Holm equation, which, in a nonlocal form, reads
\begin{equation}\label{weakCH}
u_T + uu_x + \partial_x (1 - \partial_x^2)^{-1} \LC u^2 + {1\over2}u_x^2 \RC = 0.
\end{equation}

Analytically, the Novikov equation also shares many properties in common with the Camassa-Holm equation, among which the two most remarkable features are the breaking waves and peakons. From examining the weak formulation \eqref{weakNov}, a transport theory can be applied to derive the blow-up criterion which asserts that singularities are caused by the focusing of characteristics. This in combined with the $H^1$-conservation of solutions indicates that the exact blow-up scenario is in the sense of {\it wave-breaking}, i.e., the solution remains bounded but its slope becomes infinite in finite time. Some results on this issue can be found, for instance, in \cite{CGLQ,JN}.

As an example, the wave-breaking phenomenon can also be manifested by the so-called {\it multi-peakon solutions}. The Novikov multi-peakon solution takes the form
\begin{equation}\label{intro_peakon}
u(t, x) = \sum^N_{i=1} p_i(t) e^{|x - q_i(t)|}
\end{equation}
subject to the following equations of motion for the peak positions $q_i(t)$ and amplitudes $p_i(t)$ \cite{HLS,HW}
\begin{equation}\label{HT}
\begin{cases}
\displaystyle \dot{q}_i = \sum^N_{j,k = 1} p_jp_k e^{-|q_i - q_j| - |q_i - q_k|}, \\
\displaystyle \dot{p}_i = p_i \sum^N_{j,k=1} p_jp_k \textrm{sgn}(q_i - q_j) e^{-|q_i - q_j| - |q_i - q_k|}.
\end{cases}
\end{equation}
It can be seen that a single peakon always travels to the right since $\dot{q}_i = u(q_i)^2$. Therefore for a Novikov peakon pair, only over-taking collisions may take place \cite{MK}, which differs from the Camass-Holm case where head-on collisions are also possible. In a typical situation when two peakons of anti-strength ($p_1p_2 < 0$) cross each other at time $t_*$, then as $t \to t_*^-$,
\begin{equation}\label{interaction}
\begin{split}
& p_1(t) \to +\infty\ (\text{or } -\infty), \quad p_2(t) \to -\infty\ (\text{or } +\infty), \quad p_1(t) + p_2(t) \to p_*,\\
& q_1(t), q_2(t) \to q_*, \quad q_1(t) < q_2(t) \ \text{ for } t < t_*^-,
\end{split}
\end{equation}
for some $p_*, q_* \in \mathbb{R}$; see \cite{MK}. In particular we have $\|u_x\|_{L^\infty} \to \infty$. In this case, one needs to be careful with the meaning of continuing a solution beyond a collision.

In a recent work \cite{CCL}, we managed to find a way to uniquely extend the Novikov solution beyond the point of collision, such as  multi-peakon solutions, while keeping the ``total energy" conserved. Moreover, the result in \cite{CCL} applies to the general case of continuing solution after wave-breaking. Compared to the Camassa-Holm equation, the strong nonlinearity and nonlocal effects in the Novikov equation necessitates the need to work in a higher regularity space. Thanks to a higher order conservation law \eqref{2.9}, which will serve as the ``total energy", one can close all the estimates in $H^1(\mathbb{R})\cap W^{1,4}(\mathbb{R})$. This extra regularity enhancement also results in the exact conservation of the $H^1$ norm of solutions for all time, which is in contrast to the Camassa-Holm case, where there are still possible concentration of $u_x^2 dx$ \cite{BCZ,BC2,HR}. In fact, the energy concentration in the Novikov equation occurs at the level of $u_x^4 dx$.

\subsection{Main result}\label{subsec_main}
Having established the existence and uniqueness of conservative solutions, our main goal here is to show that these solutions form a continuous semigroup. Because of the conservation laws \eqref{2.8}-\eqref{2.9}, the $H^1(\mathbb{R})\cap W^{1,4}(\mathbb{R})$ norm seems to be a natural metric for the stability theory. However, the previous multi-peakon interaction reveals the opposite. From the result of \cite[Theorem 1.1]{CCL}, for each time $t$, we introduce the measure $\nu_t$ whose absolutely continuous part with respect to Lebesgue measure has density $u^4 + 2u^2u_x^2 - {1\over3}u_x^4$. Consider a time $t_*$ when a pair of peakons collide according to \eqref{interaction}. Then as $t \to t_*^-$, $\nu_t$ converges weakly to some $\nu_{t_*}$ which contains a Dirac mass at the point $q_*$. The energy conservation implies that
\begin{equation*}
\begin{split}
\int_{\mathbb{R}} \LC u^2 + u_x^2 \RC (t_*) \ dx & = \lim_{t \to t_*^-} \int_{\mathbb{R}} \LC u^2 + u_x^2 \RC (t) \ dx, \\
\int_{\mathbb{R}} \LC u^4 + 2u^2u_x^2 - {1\over3}u_x^4 \RC (t_*) \ dx - \nu_{t_*}(\{q_*\}) & = \lim_{t\to t_*^-} \int_{\mathbb{R}} \LC u^4 + 2u^2u_x^2 - {1\over3}u_x^4 \RC (t) \ dx.
\end{split}
\end{equation*}
Therefore if we choose as from above a sequence of two-peakons of anti-strength $u^\epsilon$ defined by $u^\epsilon = u(t-\epsilon, x)$, then
\begin{equation*}
\lim_{\epsilon\to0} \LC \|u(0) - u^\epsilon(0)\|_{H^1} + \|u(0) - u^\epsilon(0)\|_{W^{1,4}} \RC = 0,
\end{equation*}
but
\begin{equation}\label{intro_0}
\begin{split}
\lim_{\epsilon\to0} & \LC {1\over 3} \|u^\epsilon_x(t_*)\|^4_{L^4} - {1\over 3} \|u_{x}(t_*)\|^4_{L^4} \RC \\
& = \nu_{t_*}(\{q_*\}) - \lim_{\epsilon\to0} \int_{\mathbb{R}} \LB \LC (u^\epsilon)^4 + 2(u^\epsilon)^2(u^\epsilon)_x^2 \RC - \LC u^4 + 2u^2u_x^2 \RC \RB (t_*) \ dx = \nu_{t_*}(\{q_*\}) > 0.
\end{split}
\end{equation}
Therefore the flow is clearly discontinuous with respect to the $W^{1,4}$ norm. Numerical evidence verifying  (\ref{intro_0}) is given in the Section \ref{sec_peakon}. 

A further message one can take from this example is that the main obstacle in establishing the Lipschitz continuous dependence on initial data lies in the concentration phenomenon of certain energy density due to possible focusing of wave fronts; see Figure \ref{concentration figure}.
\begin{figure}[hb]
  \includegraphics[scale=1.3]{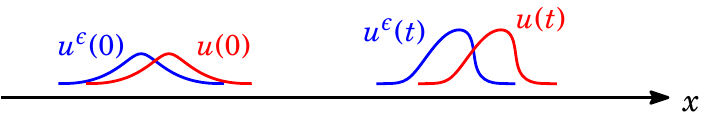}
  \caption{A sketch of the solution $u$ to the Novikov equation and a small perturbation $u^\epsilon$. Here $u^\epsilon$ is a backward shift of $u$. At time $t \to t_*^-$, both $u_x^4$ and $(u^\epsilon_x)^4$ approach to the delta functions, and hence the $W^{1,4}$ distance becomes large. A nature choice is a transport metric!}
  \label{concentration figure}
\end{figure}

Similar situation occurs in a number of other quasilinear equations, including the Hunter-Saxton equation, the Camassa-Holm equation, and the nonlinear variational wave equation. A class of optimal transport metrics was constructed to determine the minimum cost to transport an energy measure from one solution to the other, such that the corresponding flow map remains uniformly Lipschitz on bounded time intervals for this new geodesic distance; see \cite{BC2015,BF,BHR,GHR2011,GHR2013}. Compared with the previous literature, the major difficulty in dealing with the Novikov equation lies in the fact that equation \eqref{Novikov} exhibits cubic nonlinearity, which requires one to work in a stronger regularity space, and therefore changes the energy concentration nature. Hence one needs to accordingly adjust terms in the metric of the infinitesimal tangent vectors. On the other hand, two complicated non-local convolution terms make the Lipschitz estimate fairly subtle.

The main theorem in this paper can be stated as:

\begin{Theorem}\label{thm_Lip geo}
The geodesic distance $d$, defined in Definition \ref{Definition 7.2}, renders Lipschitz continuous the flow generated by the equation \eqref{2.1}. In particular, consider two initial data $u_0(x)$ and $\tilde{u}_0(x)$ which are absolute continuous and belong to $H^1(\R)\cap W^{1,4}(\R)$. Then for any $T>0$, the corresponding solutions $u(t, x)$ and $\tilde{u}(t,x)$ satisfy
\begin{equation*}
d\big(u(t),\tilde{u}(t)\big)\leq C\  d(u_0,\tilde{u}_0),
\end{equation*}
when $t\in[0,T]$, where the constant $C$ depends only on  $T$ and $H^1(\R) \cap W^{1,4}(\R)$--norm of initial data.
\end{Theorem}

\subsection{Structure of the article}\label{subsec_plan}
We construct the metric $d$ and prove that the solution map is Lipschitz continuous under this metric in several steps.

\begin{itemize}
\item[1.] We construct a Lipschitz metric for smooth solution (Sections \ref{sec_tangent}).
\item[2.] We prove by an application of the Thom's Transversality theorem that the piecewise smooth solutions with only generic singularities are dense in $H^1\cap W^{1,4}$ (Section \ref{sec_generic regularity}).
\item[3.] We first extend the Lipschitz metric to piecewise smooth solutions with only generic singularities, and then to general weak solution by the ``density" result established in Step 2 (Section \ref{sec_metric}).
\item[4.] We compare  our metric with other metrics, such as the Sobolev metric and Kantorovich-Rubinstein or Wasserstein metric (Section \ref{sec_comparison}).
\end{itemize}

A more detailed explanation is given below.

\medskip

\paragraph{\em Step 1: Metric for smooth solutions}
We begin, in Section \ref{sec_pre}, by collecting several important local and global conservation laws of equation \eqref{Novikov}, which motivates the definition of an energy conservative solution. Then we recall the main theorem on the existence and uniqueness of such solutions established in \cite{CCL}, cf. Theorem \ref{thm_exist}.

To keep track of the cost of the energy transportation, we are led to construct the geodesic distance. That is, for two given solution profiles $u(t)$ and $u^\epsilon(t)$, we consider all possible smooth deformations/paths $\gamma^t: \theta \mapsto u^\theta(t)$ for $\theta\in [0,1]$ with $\gamma^t(0) = u(t)$ and $\gamma^t(1) = u^\epsilon(t)$, and then measure the length of these paths through integrating the norm of the tangent vector $d\gamma^t/d\theta$; see Figure \ref{homotopy figure} (a). The distance between $u$ and $u^\epsilon$ will be calculated by the optimal path length
\begin{equation*}
d\LC u(t), u^\epsilon(t) \RC = \inf_{\gamma^t}\|\gamma^t\| := \inf_{\gamma^t}\int^1_0 \| v^\theta(t) \|_{u^\theta(t)} \ d\theta, \quad \text{where } v^\theta(t) = {d\gamma^t\over d\theta}.
\end{equation*}
\begin{figure}[h]
  \includegraphics[page=5, scale=1]{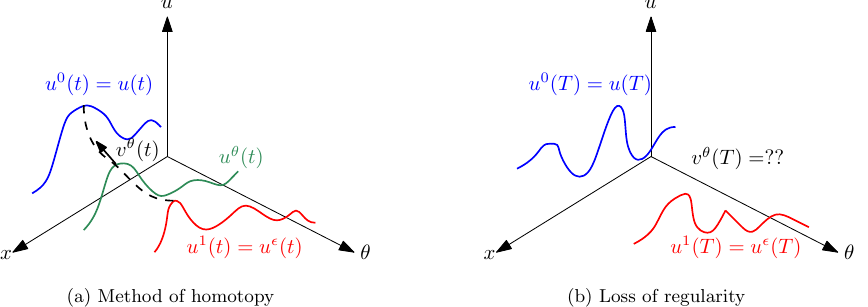}
  \caption{Compare two solutions $u(x)$ and $u^\epsilon(x)$ at a given time $t$.}
  \label{homotopy figure}
\end{figure}

Here the subscript $u^\theta(t)$ emphasizes the dependence of the norm on the flow $u$.

The next step is to define a Finsler norm for an infinitesimal tangent vector, by measuring the cost (with energy density $\mu$) in shifting from one solution to the other one. From Figure \ref{2storder figure}, in measuring the cost of transporting $u$ to $u^\epsilon$ on the $x$-$u$ plane, we notice that the tangent flow $v$ only measures the vertical displacement (dashed arrow) between two solutions. In order to provide enough freedom for  planar transports, one needs to add a quantity, named as $w$, to measure the (horizontal) shift on $x$. It is nature to consider both vertical and horizontal shifts (solid arrows) in the energy space to estimate the cost of transport.

\begin{figure}[h]
  \includegraphics[page=2, scale=0.9]{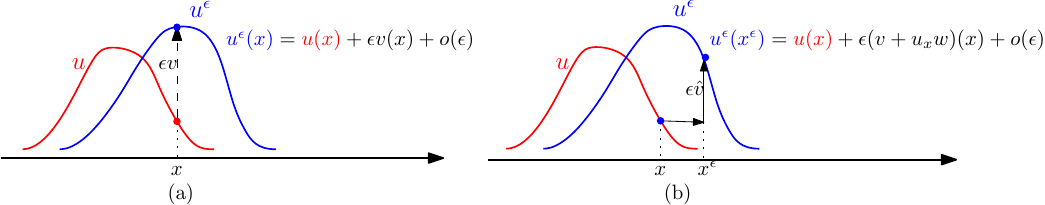}
  \caption{A sketch of how to deform from $u$ to $u^\epsilon$: (a) a vertical shift $\epsilon v$; (b) a horizontal shift $\epsilon u_x w$ followed by a vertical displacement $\epsilon v$. We denote the total shift as $\epsilon \hat v=\epsilon (v+u_x w)$ as in \eqref{1.7}. Here $x^\epsilon:=x+\epsilon w(x)$. }
  \label{2storder figure}
\end{figure}

With this in mind,, the cost function basically includes
\begin{equation}\label{1.6}
\begin{split}
\|(w, \hat v)\|_u=&\int_\R \LCB [\text{horizontal shift}]+[\text{vertical shift}]+[\text{change in } u]\RCB\,d\mu\\
&\qquad+\int_\R [\text{change in the base measure with density } \mu]\,dx,
\end{split}
\end{equation}
where, for example, as shown in Figure \ref{2storder figure}, it follows from Taylor's expansion that
\beq{\label{1.7}}
\begin{split}
[\hbox{change in } u] & = o(\epsilon) \hbox{ order of } (u^\epsilon(x^\epsilon) - u(x)) =  \hat v = v(x) + u_x(x) w(x).
\end{split}
\eeq
A more detailed description is given in Section \ref{sec_tangent}.

This way for an infinitesimal tangent vector $v$, we can accordingly define its Finsler norm to be 
\[
\|v\|_u = \inf_{w\in \mathcal{A}} \|(w, \hat v)\|_u, \quad \text{ with }\quad \hat v = v + u_x w, 
\]
where the admissible set $\mathcal{A}$ (defined in (\ref{2.16})) involves information of the characteristics of the equation that helps select the ``reasonable" transports. 



The explicit definition of the Finsler norm is provided in \eqref{abs Finsler} and \eqref{2.17}. The next thing we do is to investigate how this norm of the tangent vector propagates in time along any solution. The key estimate on the growth of the norm of the tangent vector is given in Lemma \ref{lem_est}.

The implication of Lemma \ref{lem_est} is that, for any $T>0$, if all solutions $u^\theta(t, \cdot)$ are sufficiently regular for $t\in [0, T]$, such that the tangent vectors are well-defined on $\gamma^t$, then it holds that
\begin{equation}\label{rough est}
\|\gamma^t\| \leq C \|\gamma^0\| \qquad \forall\ t \in [0, T],
\end{equation}
where $C$ only depends on the total energy of the initial data. Therefore it is natural to define the geodesic distance
\begin{equation*}
d \LC u, u^\epsilon \RC := \inf \LCB \|\gamma^t\|; \ \gamma^t:\ [0, 1] \to H^1\cap W^{1, 4}, \ \gamma^t(0) = u, \ \gamma^t(1) = u^\epsilon \RCB,
\end{equation*}
and one can expect from \eqref{rough est} that the solution map is Lipschitz continuous under this metric.
\vspace{.2cm}

\paragraph{\em Step 2: Generic regularity}
However, smooth solutions do not always remain smooth for all time. In fact for the Novikov equation \eqref{Novikov}, the gradient of the solution can blow up in finite time, and therefore a smooth path of initial data $\gamma^0$ may lose regularity at later time $T$ so that the tangent vector $d\gamma^T/dt$ may not be well-defined (see Figure \ref{homotopy figure} (b)); even if it does exist, it is not obvious that the estimate \eqref{rough est} should remain valid. The idea is to show that there are sufficiently many paths which are initially smooth, and remain (piecewise) smooth later in time. We prove in Section \ref{sec_generic regularity}
\begin{Theorem}\label{thm_generic}
Let $T>0$ be given, then there exists an open dense set 
\[\mathcal{D}\subset \Big(\mathcal{C}^3(\mathbb{R})\cap H^1(\mathbb{R})\cap W^{1,4}(\mathbb{R})\Big),\]
such that, for $u_0\in \mathcal{D}$, the conservative solution $u=u(t,x)$ of \eqref{2.1} is twice continuously differentiable in the complement of finitely many characteristic curves, within the domain $[0,T]\times \mathbb{R}.$
\end{Theorem}

The generic regularity is itself of great interest because it shows very detailed structures of singularities. As a consequence, in Corollary \ref{Corollary 4.1}, we prove the existence of regular paths  connecting any two solutions of Theorem \ref{thm_generic}.

Similar generic regularity result was first established for the variational wave equation \cite{BC} and used in establishing a Lipschitz metric in \cite{BC2015}. This idea was applied to the Camassa-Holm equation in \cite{LZ}. The proof of Theorem \ref{thm_generic} relies on an application of Thom's Transversality Theorem.
\vspace{.2cm}

\paragraph{\em Step 3: Metric for general weak solutions}
For general weak solutions, we first use the semi-linear system established in \cite{CCL} to extend the Lipschitz metric from smooth solutions to piecewise smooth solutions, cf. Theorem \ref{thm_length}, Definition \ref{def5.1} and Definition \ref{Definition 7.1}. In the second step, we generalize this metric on the space $H^1(\R)\cap W^{1,4}(\R)$ and prove the Lipschitz property. This step is now very nature because the set of regular paths, that are roughly speaking
piecewise smooth solutions, is a dense set of the solution space.
\vspace{.2cm}

%
\paragraph{\em Step 4: Compare with other metrics} See details in Section 6.

\begin{Remark}
Currently, there are several parallel techniques in constructing the Lipschitz metrics.
These include: the Finsler distance used in this paper (first established in \cite{BC,BC2015} for variational wave equation);  method in \cite{GHR2011,GHR2013} for Camassa-Holm equation, which uses semi-linear equations on Lagrangian coordinates; and a direct transport method in \cite{BF} combing a convergence argument from multi-peakon solutions.

Notice that, since we mainly work on the original coordinates and do not need the existence of sufficiently many special solutions such as multi-peakon solutions, our way to construct the metric is quite transparent and robust, as long as one has good knowledge about the energy concentration. In Section \ref{appendix CH}, we also apply this method to the Camassa-Holm equation, and the computation for the key estimates is greatly simplified. 
\end{Remark}

We also point out at the end of the article that the semi-linear system introduced in \cite{CCL} can be successfully implemented to simulate the peakon interactions. This is usually considered to be a quite challenging problem. Most of the current methods treat directly the ODEs governing the dynamics of peakons. 
Yet to the authors' knowledge, no schemes have been proposed to realize the energy concentration exactly at the collision time. This is largely due to the blow-up of the momentum density when two peakons collide. Here we take advantage of the semi-linear system in \cite{CCL} which makes use of a set of new variables that dilates the interacting characteristics and in particular, allows one to trace the behavior of peakons even at the instant of collision. We present in Section \ref{sec_peakon} several numerical experiments to both showcase explicitly the energy concentration of a peakon-antipeakon interaction, and the merits of using our scheme. Our method can be easily employed to other integrable system to analyze the peakon interactions.

%

\section{Preliminary results}\label{sec_pre}
In this section we recall some useful estimates and the existence and uniqueness of global conservative solution of the Novikov equation. For more details the reader can refer to \cite{CCL}.

First we rewrite the Cauchy problem for the Novikov equation \eqref{Novikov} as
 \begin{equation}\label{2.1}
\begin{cases}
       u_t+u^2u_x+\partial_x P_1+P_2=0,\\
       u(0,x)=u_0(x),
        \end{cases}
 \end{equation}
where 
\begin{equation}\label{2.2}
P_1:=p*\LC \frac{3}{2}uu_x^2+u^3 \RC ,\quad P_2:=\frac{1}{2}p*\LC u_x^3 \RC,
\end{equation}
with $p=\frac{1}{2}e^{-|x|}$ being the Green's function for $(1 - \partial_x^2)^{-1}$ on $\R$.

For smooth solutions, differentiating equation \eqref{2.1} with respect to $x$ we have
 \begin{equation}\label{2.3}
       u_{xt}+u^2u_{xx}+\frac{1}{2}uu_x^2-u^3+P_1+\partial_x P_2=0.
 \end{equation}
Further manipulation leads to the following local conservation laws
\begin{align}\label{2.6}
      & \LC \frac{u^2+u_{x}^2}{2} \RC_t+\LC \frac{u^2u^2_{x}}{2}+uP_1+u\partial_x P_2 \RC_x=0, \\
\label{2.7}
 \begin{split}
      & \LC u^4+2u^2u_x^2-\frac{1}{3}u_x^4 \RC_t+ \LC 2u^4u_x^2-\frac{1}{3}u^2u_x^4+\frac{4}{3}u^3(P_1+\partial_x P_2)\RC_x\\
       &\qquad+\frac{4}{3}\LC(P_1+\partial_x P_2)^2-(P_2+\partial_x P_1)^2\RC_x=0,
  \end{split}
\end{align}
which indicates two conserved quantities
\begin{align}\label{2.8}
    \mathcal{E}(t) & =\int_{\R}  (u^2+u_{x}^2)(t,x)\,dx=\mathcal{E}(0) ,\\
\label{2.9}
    \mathcal{F}(t) & =\int_{\R}  (u^4+2u^2u_x^2-\frac{1}{3}u_x^4)(t,x)= \mathcal{F}(0) \,dx.
 \end{align}
Therefore we can bound
\begin{align*}
\|u\|_{L^\infty}^2 & \leq \|u\|_{H^1}^2=\mathcal{E}(0),\\
\|u_x\|_{L^4}^4 & =3\int_\R (u^4+2u^2u_x^2)\,dx-3\mathcal{F}(t)
    \leq 3\LC 2\mathcal{E}(0)^2-\mathcal{F}(0) \RC,
\end{align*}
which in turn implies that
   \begin{equation}\label{2.11}
    \|u_x\|_{L^3}^3\leq \sqrt{3\mathcal{E}(0)\big(2\mathcal{E}(0)^2-\mathcal{F}(0)\big)}=: K.   \end{equation}
 Now, we are able to bound $P_i(t)$ and the derivatives $\partial P_i$ for $i=1,2$ as follows.
   \begin{equation}\label{2.12}
  \begin{split}
 & \|P_1(t)\|_{L^\infty},  \|\partial_x P_1(t)\|_{L^\infty}\leq \|p\|_{L^\infty}\|\frac{3}{2}uu_x^2+u^3\|_{L^1}\leq \frac{3}{4}\mathcal{E}(0)^\frac{3}{2},\\
   & \|P_1(t)\|_{L^2},  \|\partial_x P_1(t)\|_{L^2}\leq \|p\|_{L^2}\|\frac{3}{2}uu_x^2+u^3\|_{L^1}\leq \frac{3}{2\sqrt{2}}\mathcal{E}(0)^\frac{3}{2},\\
   & \|P_2(t)\|_{L^\infty},  \|\partial_x P_2(t)\|_{L^\infty}\leq \frac{1}{2}\|p\|_{L^\infty}\|u_x^3\|_{L^1}\leq \frac{1}{4}K,\\
   & \|P_2(t)\|_{L^2},  \|\partial_x P_2(t)\|_{L^2}\leq \frac{1}{2}\|p\|_{L^2}\|u_x^3\|_{L^1}\leq \frac{1}{2\sqrt{2}}K.
    \end{split} \end{equation}
    
Now we state the theorem on the existence and uniqueness of conservative solutions to the Cauchy problem \eqref{2.1}.
\begin{Theorem}[\cite{CCL}]\label{thm_exist}
Let $u_0\in H^1(\R)\cap W^{1,4}(\R)$ be an absolute continuous function. Then the Cauchy problem \eqref{2.1} admits a unique energy {\bf\em conservative} solution $u(t,x) $ defined for all $(t,x)\in \R^+\times\R$ in the following sense.
\begin{itemize}
\item[(i)] For any fixed $t\geq 0$, $u(t,\cdot)\in H^1{(\R)}\cap W^{1,4}{(\R)}$.
The map $t\rightarrow u(t,\cdot)$ is Lipschitz continuous under the $L^4(\R)$ metric.

\item[(ii)] The solution $u=u(t,x)$ satisfies the initial condition of \eqref{2.1} in $L^4(\R),$ and
\beq\label{nv_weak}
\iint_{\Lambda}
	\left\{-u_x\, \bigl(\phi_t+u^2\,\phi_x)+ \LC -\frac{3}{2}u u_x^2-u^3+P_1 +\partial_x P_2\RC \phi\right\}\, dx\, dt +\int_{\R}u_{0,x}\phi(0,x)\,dx=0
\eeq
for every test function $\phi\in C_c^1(\Lambda)$ with $\Lambda=\Big\{(t,x)\,\Big|\  t\in [0, \infty), x\in \R\Big\}$.

\item[(iii)] There exists a family of Radon measures $\{\mu_t,\, t\in \R^+\}$,  depending continuously on time and w.r.t the topology of weak convergence of measures.
For every $t\in \R^+$,  the absolutely continuous part of $\mu_t$ with respect to the Lebesgue measure has density $u_x^4(t,\cdot)$, which provides a
measure-valued solution to the balance law
\bel{weak_en}
 \int_{\R^+} \left\{ \int(\phi_t+u^2\phi_x)d\mu_t+\int\Big(4u^3 u_x^3-4u_x^3 (P_1+\partial_x P_2)\Big)\phi\, dx \right\}dt - \int_\R u^4_{0,x}\phi(0,x)dx=0,
\eeq
for every test function $\phi\in C_c^1(\Lambda)$.
\end{itemize}
Moreover, the solution also satisfies the following properties.
\begin{itemize}
\item[(1)] $u(t,x)$ is H\"older continuous with exponent $3/4$ on both $t$ and $x$.
\item[(2)] The first energy density $u^2+u_x^2$ is conserved for any time $t\geq 0$, i.e.
\bel{energy1}
\mathcal{E}(t)=\| u(t)\|^2_{H^1}=\| u_0\|^2_{H^1}\quad\hbox{for any}\quad t\geq0;
\eeq

\item[(3)] The second energy density $u^4 + 2u^2u^2_x - {1\over3}u^4_x$ is conserved in the following sense.
\begin{itemize}
\item[(a)] An energy inequality is satisfied: 
\bel{energy2}
\mathcal{F}(t) =\int_\R \LC u^4 + 2u^2u^2_x - {1\over3}u^4_x \RC(t,x)\ dx \geq \mathcal{F}(0)\quad\hbox{for any}\quad t\geq0.
\eeq
\item[(b)]  Denote  a family of Radon measures $\left\{\nu_t,\  t\in \R^+\right\}$, such that
\[\nu_t(\mathcal{A})=\int_\mathcal{A} \LC u^4 + 2u^2u^2_x  \RC(t,x)\ dx-{1\over3}\mu_t(\mathcal{A})\]
for any  Lebesgue  measurable set $\mathcal{A}$ in $\R$. Then
for any $t\in \R^+$,
\[
\nu_t(\R)=\nu_0(\R)=\mathcal{F}(0)=\int_\R \LC u^4 + 2u^2u^2_x - {1\over3}u^4_x \RC(0,x)\ dx.
\]

For any $t\in \R^+$,  the absolutely continuous part of $\nu_t$ with respect to Lebesgue measure has density $u^4 + 2u^2u^2_x - {1\over3}u^4_x$.
For almost every $t\in \R^+$, the singular part of $\nu_t$ is concentrated on the set where
$u=0$.
\end{itemize}
\item[(4)]
A continuous dependence result holds. Consider a sequence of initial data ${u_0}_n$ such that
$\|{u_0}_n-u_0\|_{H^1\cap W^{1,4}}\rightarrow 0$, as $n\rightarrow\infty$. Then
the corresponding solutions $u_n(t,x)$ converge to $u(t,x)$ uniformly for $(t,x)$
in any bounded sets.
\end{itemize}
\end{Theorem}
 
\section{The norm of tangent vector for smooth solutions}\label{sec_tangent}
To illustrate the ideas on how to construct the Lipschitz metric, in this section we first consider smooth solutions to \eqref{2.1}. We take a family of perturbed solutions $u^\epsilon(x)$ to \eqref{2.1}, which can be written as
\begin{equation}\label{2.13}
u^\epsilon(x)=u(x)+\epsilon v(x)+o(\epsilon).
\end{equation}
Here in this section, we make an abuse of notation of using $f(t)$ or $f(x)$ to denote a function $f(t,x)$. 

A straightforward calculation yields that the first order perturbation $v$ satisfies
\begin{equation}\label{2.14}
\begin{split}
v_t+u^2v_x&+2uvu_x+\frac{1}{2}\LC \int_x^{\infty}-\int^x_{-\infty}\RC e^{-|x-y|}\LC 3u_y v_y u+\frac{3}{2}u^2_y v+3u^2v \RC \,dy\\
&+\frac{3}{4}\int^{\infty}_{-\infty}e^{-|x-y|}u^2_yv_y\,dy=0.
\end{split}\end{equation}
Differentiating \eqref{2.14} with respect to $x$, one obtains
\begin{equation}\label{2.15}\begin{split}
v_{xt}+u^2v_{xx}&+uu_xv_x+\frac{1}{2}u^2_xv +2uvu_{xx}-3u^2v+\frac{1}{2}\int_\R e^{-|x-y|} \LC 3u_y v_y u+\frac{3}{2}u^2_y v+3u^2v \RC \,dy\\
&+\frac{3}{4}\LC \int_x^{\infty}-\int^x_{-\infty}\RC e^{-|x-y|}u^2_yv_y\,dy=0.
\end{split}\end{equation}

As explained in Section \ref{subsec_plan}, to obtain enough freedom in measuring the shift from one solution to the other one, we need to add a quantity $w(t,x)$ measuring the horizontal shift:
\beq{\label{wdef}}
x^{\epsilon}:=x+\epsilon w(x) +o(\epsilon).
\eeq
To focus only on reasonable transports between two solutions, we select $w(t,x)$ by propagating along characteristics the shifts $w_0(x)$ as the initial data.
%
That is, we require that when $x(t)$ is a characteristic emanating from $x_0$ then $x^\epsilon(t)$ is also a characteristic emanating from $x^\epsilon_0$, so
\begin{equation*}
\frac{d}{dt}x^\epsilon(t)= \LC u^\epsilon(x^\epsilon) \RC^2 \qquad{\hbox{when}}\qquad
\frac{d}{dt}x(t)= u^2(x).
\end{equation*}
Thus, using (\ref{2.13}), (\ref{wdef}) and taking the limit as $\epsilon\rightarrow 0$, we have 
\begin{equation}\label{2.18}
w_t+u^2w_x=2u(v+u_x w).
\end{equation}

So the Finsler norm on the space of infinitesimal tangent vector $v$ takes the form
\begin{equation}\label{Finsler v}
\|v\|_u: = \inf_{w\in \mathcal{A}} \|(w, \hat v)\|_u, \quad \text{ with }\quad \hat v = v + u_x w, 
\end{equation}
where the admissible set is defined as
\begin{equation}\label{2.16}
\qquad\mathcal{A} = \LCB {{\hbox{solutions $w(t,x)$ of \eqref{2.18} with  smooth initial data $w_0(x)$} }}  \RCB.
\end{equation}
Note that $u$ is always smooth in this section, hence in (\ref{2.18}) $w$ can be solved with a given initial data  $w(0,x) = w_0(x)$.


To motivate the explicit construction of $\|(w,v)\|_u$, we consider a reference solution $u$ together with a perturbation $u^\epsilon$, as shown in Figure \ref{2storder figure}. Recall that we are interested in determining the cost of transporting the energy (with density $\mu \approx u_x^4$) from $u$ to $u^\epsilon$. In fact we will choose $\mu = (1 + u_x^2)^2$, and the cost should account for the following
\begin{equation}\label{abs Finsler}
\begin{split}
\|(w, \hat{v})\|_u & =\int_\R \LCB [\text{change in } x]+[\text{change in } u]+[\text{change in  arctan} u_x]\RCB(1+u_x^2)^2 e^{-|x|}\,dx\\
& \quad +\int_\R [\text{change in the base measure with density } (1+u_x^2)^2] e^{-|x|}\,dx\\
& =: I_1+I_2+I_3+I_4\,.
\end{split}
\end{equation}

More precisely, we have
\begin{equation}\label{2.17}
\begin{split}
\|v\|_u
&=\inf_{w \in \mathcal{A}}\int_\R \LCB |w|(1+u_x^2)^2+|v+u_x w|(1+u_x^2)^2+|v_x+u_{xx}w|(1+u_x^2) \right.\\
&\qquad\qquad \left. +|4(u_x+u_x^3)(v_x+u_{xx}w)+(1+u_x^2)^2w_x|\RCB\, e^{-|x|}\,dx\\
&=: \inf_{w \in \mathcal{A}} \LC I_1+I_2+I_3+I_4 \RC.
\end{split}\end{equation}
We note that when $u\in L^1$, the $e^{-|x|}$ term is not necessary.
Below we briefly explain how to obtain \eqref{2.17}, using (\ref{2.13}) and (\ref{wdef}).
\begin{itemize}
\item For [change in $x$] in $I_1$,
\[
{1\over \epsilon}\LB x^\epsilon - x \RB=w+o(\epsilon).
\]
\item For [change in $u$] in $I_2$,
\[
{1\over \epsilon}\LB u^\epsilon(x^\epsilon) - u(x) \RB = v(x) + u_x(x) w(x)+ o(\epsilon).
\]
\item For [change in $\arctan u_x$] in $I_3$, 
\begin{equation*}
\begin{split}
\arctan u^\epsilon_x(x^\epsilon) & = \arctan \LB u_x(x^\epsilon) + \epsilon v_x(x^\epsilon) + o(\epsilon) \RB \\
& = \arctan \LB u_x(x) + \epsilon w(x) u_{xx}(x) + \epsilon v_x(x) + o(\epsilon) \RB \\
& = \arctan u_x(x) + \epsilon {v_x(x) + w(x)u_{xx}(x) \over 1 + u_x^2(x)}+ o(\epsilon) .
\end{split}
\end{equation*}
\item For  [change in the base measure with density $(1+u_x^2)^2$] in $I_4$, using the following identities
\begin{equation*}
\begin{split}
\LC u^\epsilon_x(x^\epsilon) \RC^2 & = u_x^2(x^\epsilon) + 2\epsilon u_x(x^\epsilon) v_x(x^\epsilon) + o(\epsilon) \\
& = u_x^2(x) + 2\epsilon w(x) u_x(x) u_{xx}(x) + 2\epsilon v_x(x) u_x(x) + o(\epsilon),
\end{split}
\end{equation*}
we obtain that 
\begin{equation*}
\begin{split}
\LC 1 + \LC u^\epsilon_x(x^\epsilon) \RC^2 \RC^2dx^\epsilon - \LC 1 + u_x^2(x) \RC^2  dx
= \LC 4\epsilon (v_x + w u_{xx})(u_x + u_x^3) + \epsilon w_x (1 + u_x^2)^2 + o(\epsilon)\RC dx.
\end{split}
\end{equation*}
\end{itemize}

Next we want to understand how this norm of the tangent vector changes in time. The main result of this section is the following.
\begin{Lemma}\label{lem_est} Let $T>0$ be given, 
and $u(t,x)$ be a smooth solution to \eqref{2.1} when $t\in[0,T]$. Assume that the first order perturbation $v$ satisfies equation \eqref{2.14}. Then it follows that 
\begin{equation}\label{2.19}
\|v(t)\|_{u(t)}\leq \mathcal{C}(T)\|v(0)\|_{u(0)},
\end{equation}
for some constant $\mathcal{C}(T)$ depending only on initial total energy $\mathcal{E}(0)$, $\mathcal{F}(0)$ and $T$.
\end{Lemma}


\begin{proof}
It suffices to show that
\begin{equation}\label{est on w and tilde v}
{d \over dt} \LN \LC w(t), v(t) \RC \RN_{u(t)} \leq C \LN \LC w(t), v(t) \RC \RN_{u(t)},
\end{equation}
for  any $v$ and $w$ satisfying (\ref{2.14}), (\ref{2.15}) and (\ref{2.18}). Here $C$ is a constant depending only on the initial energy $\mathcal{E}(0)$ and $\mathcal{F}(0)$. In rest of this paper, unless specified, we will use $C$ to denote a constant depending on the initial energy $\mathcal{E}(0)$, $\mathcal{F}(0)$.

To prove \eqref{est on w and tilde v}, first, notice that for any smooth function $f$, we have
\begin{equation*}
\begin{split}
&\quad\frac{d}{dt}\int_\R |f|e^{-|x|}\,dx\\
=& \int_\R (|f|e^{-|x|})_t+(u^2 |f|e^{-|x|})_x\,dx\\
=&\int_\R \text{ sign}(f)[f_t+(u^2 f)_x]e^{-|x|}\,dx 
-\int_{0}^{\infty}u^2 |f| e^{-x}dx+\int_{-\infty}^{0}u^2 |f| e^x dx\\
\leq& \int_\R |f_t+(u^2 f)_x|e^{-|x|}\,dx +\|u\|_{L^\infty}^2\int_\R |f| e^{-|x|}dx.
\end{split}
\end{equation*}
Therefore in the following we only have to show that  
$
|f_t+(u^2 f)_x|e^{-|x|}
$ is bounded by the integrands in $I_1+I_2+I_3+I_4$,
with $f$ being $w(1+u_x^2)^2,\ (v+u_x w)(1+u_x^2)^2,\ (v_x+u_{xx}w)(1+u_x^2)$ and $4(u_x+u_x^3)(v_x+u_{xx}w)+(1+u_x^2)^2w_x$, respectively. 

To simplify the computation, we drop all the $e^{-|x|}$ terms (this can be thought of as assuming that the solution has compact support). In the general case, the same results remain valid when the factor $e^{-|x|}$ is inserted back.
\medskip

%
%
%

{\bf 1.} For $I_1$ in \eqref{2.17}, it follows from a direct computation that
\begin{equation}\label{2.20}
\big((1+u_x^2)^2\big)_t+\big(u^2(1+u_x^2)^2\big)_x=(u^3_x+u_x)(4u^3-4(P_1+\partial_x P_2)+2u).
\end{equation}
By \eqref{2.18} and \eqref{2.20}, we have
\begin{equation}\label{2.23}
\begin{split}
&\LC w(1+u_x^2)^2 \RC_t  + \LC u^2w(1+u_x^2)^2 \RC_x\\
= &\ (w_t+u^2w_x)(1+u_x^2)^2+w\LB \LC (1+u_x^2)^2 \RC_t + \LC u^2(1+u_x^2)^2 \RC_x \RB\\
= &\ 2u(v+u_xw)(1+u_x^2)^2+w(u^3_x+u_x) \LC 4u^3-4(P_1+\partial_x P_2)+2u \RC.
\end{split}
\end{equation}
This in turn yields the estimate
\begin{equation}\label{2.24}
\begin{split}
{dI_1\over dt} = \frac{d}{dt}\int_\R |w|(1+u_x^2)^2\,dx & \leq C\int_\R |v+u_xw|(1+u_x^2)^2\,dx+C\int_\R |w|(1+u_x^2)^2\,dx\\
& \leq C (I_1 + I_2).
\end{split}
\end{equation}

{\bf 2.} The estimates for the second and third terms are much more delicate. For $I_2$, from \eqref{2.3}, \eqref{2.14}, \eqref{2.18} and \eqref{2.20}, we obtain
\begin{equation}\label{2.25}
\begin{split}
& \quad\LC (v+u_xw)(1+u_x^2)^2 \RC_t+ \LC u^2(v+u_xw)(1+u_x^2)^2 \RC_x\\
& = \LB v_t+u^2v_x+u_x(w_t+u^2w_x)+w\LC u_{xt}+u^2u_{xx}\RC \RB (1+u_x^2)^2\\
&\quad +(v+u_xw)\LB \LC(1+u_x^2)^2 \RC_t+\LC u^2(1+u_x^2)^2\RC_x \RB\\
& = \LB -\frac{1}{2}\LC\int_x^{\infty}-\int^x_{-\infty}\RC e^{-|x-y|}\LC 3u_y v_y u+\frac{3}{2}u^2_y v+3u^2v\RC \,dy-\frac{3}{4}\int^{\infty}_{-\infty}e^{-|x-y|}u^2_yv_y\,dy \right.\\
&\quad\  \left.-2uvu_x+2uu_x(v+u_xw)+w\LC-\frac{uu_x^2}{2}+u^3-P_1-\partial_x P_2\RC\RB (1+u_x^2)^2\\
&\quad +(v+u_xw)(u^3_x+u_x)\LC 4u^3-4(P_1+\partial_x P_2)+2u\RC\\
& = \LB -\frac{1}{2}\LC\int_x^{\infty}-\int^x_{-\infty}\RC e^{-|x-y|}\LC 3u_y v_y u+\frac{3}{2}u^2_y v+3u^2v\RC \,dy-\frac{3}{4}\int^{\infty}_{-\infty}e^{-|x-y|}u^2_yv_y\,dy \right.\\
&\quad \ \left.+\frac{3 }{2}uu_x^2w+w(u^3-P_1-\partial_x P_2)\RB (1+u_x^2)^2+(v+u_xw)(u^3_x+u_x)(4u^3-4(P_1+\partial_x P_2)+2u)\\
& =: (I_{21}+I_{22}+I_{23}+I_{24})(1+u_x^2)^2+(v+u_xw)(u^3_x+u_x)(4u^3-4(P_1+\partial_x P_2)+2u).
\end{split}
\end{equation}

For the local term $I_{23}=\frac{3}{2}u u_x^2w(x)$, we have
\begin{equation*}
\begin{split}
I_{23}=&-\frac{1}{2}\LC \int_x^{\infty}-\int^x_{-\infty}\RC \LC e^{-|x-y|}\frac{3}{2}u u_y^2w\RC_y\,dy\\
=& -\frac{1}{2}\LC \int_x^{\infty}-\int^x_{-\infty}\RC e^{-|x-y|} \LC 3uu_yu_{yy}w+\frac{3}{2}u^3_yw+\frac{3}{2}uu^2_yw_y\RC\,dy\\
&+\frac{3}{4}\int_{-\infty}^\infty e^{-|x-y|}u u_y^2w\,dy.
\end{split}
\end{equation*} 
Hence
\begin{align*}
&\quad I_{21}+I_{23}\\
& = -\frac{1}{2}\LC \int_x^{\infty}-\int^x_{-\infty}\RC e^{-|x-y|}\LB 3uu_y(v_y+u_{yy}w)+\frac{3}{2}uu^2_yw_y+\frac{3}{2}u^2_y (v+u_yw)+3u^2v \RB \,dy\\
&\quad +\frac{3}{4}\int_{-\infty}^\infty e^{-|x-y|}u u_y^2w\,dy\\
& = -\frac{1}{2}\LC \int_x^{\infty}-\int^x_{-\infty}\RC e^{-|x-y|}\LCB \frac{3}{2}u\LB 4u_y(v_y+u_{yy}w)+(1+u_y^2)w_y\RB -3uu_y(v_y+u_{yy}w) \right. \\
& \quad \left. -\frac{3}{2}(uw)_y+\frac{3}{2}u_y w+\frac{3}{2}u^2_y (v+u_yw)+3u^2v\RCB\,dy+\frac{3}{4}\int_{-\infty}^\infty e^{-|x-y|}u u_y^2w\,dy\\
& = -\frac{1}{2}\LC \int_x^{\infty}-\int^x_{-\infty}\RC e^{-|x-y|}\LCB \frac{3}{2}u\LB 4u_y(v_y+u_{yy}w)+(1+u_y^2)w_y\RB -3uu_y(v_y+u_{yy}w) \right.\\
&\quad \left. +\frac{3}{2}u_y w+\frac{3}{2}u^2_y (v+u_yw)+3u^2v\RCB\,dy-\frac{3}{4}\LC \int_x^{\infty}-\int^x_{-\infty}\RC \frac{\partial e^{-|x-y|}}{\partial y}uw\,dy\\
&\quad -\frac{3}{2}uw(x)+\frac{3}{4}\int_{-\infty}^\infty e^{-|x-y|}u u_y^2w\,dy\\
& = -\frac{1}{2}\LC \int_x^{\infty}-\int^x_{-\infty}\RC e^{-|x-y|}\LCB \frac{3}{2}u\LB 4u_y(v_y+u_{yy}w)+(1+u_y^2)w_y \RB -3uu_y(v_y+u_{yy}w) \right. \\
&\quad \left. +\frac{3}{2}u_y w+\frac{3}{2}u^2_y (v+u_yw)+3u^2v\RCB\,dy+\frac{3}{4}\int_{-\infty}^\infty e^{-|x-y|}(uw+u u_y^2w)\,dy-\frac{3}{2}uw(x).
\end{align*}

Note that
 \begin{equation*}
3u^2v(y)=3u^2(v+u_yw)-3u^2u_yw,
\end{equation*}
which, together with the previous equality and the Sobolev inequality, implies that
\begin{equation}\label{2.27}
\begin{split}
&\quad \LV \int_\R (I_{21}+I_{23})(1+u_x^2)^2\,dx \RV \\
& \leq C\int_\R \LC \int_\R e^{-|x-y|}(1+u_x^2)^2\,dx\RC \LB \LV \frac{3}{2}u \LC 4u_y(v_y+u_{yy}w)+(1+u_y^2)w_y \RC \RV \right.\\
&\quad + \LV 3uu_y(v_y+u_{yy}w) \RV + \frac{3}{2}|u_y w|+\frac{3}{2}|u^2_y (v+u_yw)|+3|u^2(v+u_yw)|+3|u^2u_yw| \\
&\quad \left. +\frac{3}{4}|uw|+\frac{3}{4}|u u_y^2w|\RB \,dy+C\int_\R |w|(1+u_x^2)^2\,dx\\
& \leq C\int_\R \LV 4u_x(v_x+u_{xx}w)+(1+u_x^2)w_x \RV (1+u_x^2)\,dx+C\int_\R |w|(1+u_x^2)^2\,dx\\
&\quad +C\int_\R |v_x+u_{xx}w|(1+u_x^2)\,dx+C\int_\R |v+u_xw|(1+u_x^2)^2\,dx\\
& \leq C (I_4 + I_1 + I_3 + I_2),
\end{split}
\end{equation}
where we have used the fact that 
\[\int_\R e^{-|x-y|}(1+u_x^2)^2\,dx\leq \int_\R e^{-|x-y|}\,dx+2\int_\R u_x^2\,dx+\int_\R u_x^4\,dx\leq C.\]

For the second term $I_{22}$, we have
\begin{equation*}
\begin{split}
3u^2_yv_y=3u^2_y(v_y+u_{yy}w)-3u_y^2u_{yy}w=3u^2_y(v_y+u_{yy}w)-(u_y^3w)_y+u_y^3w_y.
\end{split}\end{equation*}
The last term can be bounded as
\begin{equation}\label{2.28}
\begin{split}
u_y^3w_y&\leq |(1+u_y^2)u_yw_y|\leq |4u_y^2(v_y+u_{yy}w)+(1+u_y^2)u_yw_y|+|4u_y^2(v_y+u_{yy}w)|\\
&\leq C |4u_y(v_y+u_{yy}w)+(1+u_y^2)w_y|(1+u_y^2)+|4u_y^2(v_y+u_{yy}w)|.
\end{split}
\end{equation} 
Then it holds that
\begin{align}
&\quad\LV \frac{3}{4}\int_\R \int_\R e^{-|x-y|}u^2_yv_y\,dy(1+u_x^2)^2\,dx\RV \nonumber\\
& \leq C \int_\R \LC \int_\R e^{-|x-y|}(1+u_x^2)^2\,dx \RC \LC |u^2_y(v_y+u_{yy}w)|+|4u_y(v_y+u_{yy}w)+(1+u_y^2)w_y|(1+u_y^2)\RC \,dy \nonumber\\
& \quad + C \LV \int_\R \LC \int_\R e^{-|x-y|}(u_y^3w)_y\,dy \RC(1+u_x^2)^2\,dx \RV \label{2.29} \\
& \leq C \int_\R |v_x+u_{xx}w|(1+u_x^2)\,dx+C \int_\R \LV 4u_x(v_x+u_{xx}w)+(1+u_x^2)w_x \RV(1+u_x^2)\,dx \nonumber\\
&\quad +C \int_\R\int_\R \LV \frac{\partial e^{-|x-y|}}{\partial y} \RV(1+u_x^2)^2\,dx \cdot |u_y^3w|\,dy \nonumber\\
& \leq C (I_3 +I_4 +I_1). \nonumber
\end{align}
Plugging \eqref{2.27} and \eqref{2.29} into \eqref{2.25}, we can conclude that
\begin{equation}\label{2.30}
\begin{split}
{dI_2 \over dt} & = \frac{d}{dt}\int_\R |v+u_xw|(1+u_x^2)^2\,dx \leq C(I_1 + I_2 + I_3 + I_4). 
\end{split}
\end{equation}

{\bf 3.} To estimate the time derivative of $I_3$, using \eqref{2.3}, \eqref{2.15} and \eqref{2.18} we obtain
\begin{equation}\label{2.31}
\begin{split}
&\quad\LC (v_x+u_{xx}w)(1+u_x^2)\RC_t+\LC u^2(v_x+u_{xx}w)(1+u_x^2)\RC_x\\
& =\LB v_{xt}+u^2v_{xx}+u_{xx}(w_t+u^2w_x)+w(u_{xxt}+u^2u_{xxx})\RB (1+u_x^2)\\
&\quad \ \  +(v_x+u_{xx}w)\LB (1+u_x^2)_t+\big(u^2(1+u_x^2)\big)_x\RB\\
& =\LB -uu_xv_x-\frac{1}{2}u_x^2v-2uvu_{xx}+3u^2v-\frac{1}{2}\int^{\infty}_{-\infty}e^{-|x-y|}(3u_y v_y u+\frac{3}{2}u^2_y v+3u^2v)\,dy \right. \\
&\quad \ \ -\frac{3}{4}\LC\int_x^{\infty}-\int^x_{-\infty}\RC e^{-|x-y|}u^2_yv_y\,dy+2uu_{xx}(v+u_x w)+w(3u^2u_x-3uu_xu_{xx}\\
&\quad \ \ -\partial_x P_1-P_2)\Big](1+u_x^2)+(v_x+u_{xx}w)\LB 2u^3u_x-2u_x(P_1+\partial_x P_2)+uu_x^3+2uu_x \RB\\
& = \LB-\frac{1}{2}u_x^2v+3u^2(v+u_xw)-\frac{1}{2}\int^{\infty}_{-\infty}e^{-|x-y|}(3u_y v_y u+\frac{3}{2}u^2_y v+3u^2v)\,dy \right.\\
&\left. \quad \ \ -\frac{3}{4}\LC \int_x^{\infty}-\int^x_{-\infty}\RC e^{-|x-y|}u^2_yv_y\,dy-w(\partial_x P_1+P_2)\RB (1+u_x^2)\\
&\quad +(v_x+u_{xx}w)\LB 2u^3u_x-2u_x(P_1+\partial_x P_2)+uu_x \RB.
\end{split}
\end{equation}
Note that 
\begin{equation*}
\begin{split}
-\frac{1}{2}u_x^2v&=-\frac{1}{2}u_x^2(v+u_xw)+\frac{1}{2}u_x^3w\\
&=-\frac{1}{2}u_x^2(v+u_xw)-\frac{1}{2}\LC \int_x^{\infty}-\int^x_{-\infty}\RC \LC e^{-|x-y|}\frac{1}{2}u_y^3w(y)\RC_y\,dy\\
&=-\frac{1}{2}u_x^2(v+u_xw)+\frac{1}{4}\int_{-\infty}^\infty e^{-|x-y|}u_y^3w\,dy\\
&\quad-\frac{1}{4}\LC \int_x^{\infty}-\int^x_{-\infty}\RC e^{-|x-y|}(3u_y^2u_{yy}w+u_y^3w_y)\,dy.
\end{split}
\end{equation*}
From the above inequality and \eqref{2.28}, the first and fourth terms in the last equality of \eqref{2.31} can be estimated as
\begin{align}
&\quad\LV \int_\R \LCB -\frac{1}{2}u_x^2v-\frac{3}{4}\LC \int_x^{\infty}-\int^x_{-\infty}\RC e^{-|x-y|}u^2_yv_y(y)\,dy\RCB(1+u_x^2)\,dx \RV \nonumber\\
&= \LV -\frac{1}{2}\int_\R u_x^2(v+u_xw)(1+u_x^2)\,dx+\frac{1}{4}\int_\R\int_\R e^{-|x-y|}u_y^3w\,dy(1+u_x^2)\,dx \right.\nonumber\\
&\quad \ \ \left. -\frac{1}{4}\int_\R \LC \int_x^{\infty}-\int^x_{-\infty}\RC e^{-|x-y|}[3u_y^2(v_y+u_{yy}w)+u_y^3w_y]\,dy(1+u_x^2)\,dx \RV \nonumber\\
&\leq C \int_\R |v+u_xw|(1+u_x^2)^2\,dx+\frac{1}{4}\int_\R \LC \int_\R e^{-|x-y|}(1+u_x^2)\,dx \RC |u_y^3w|\,dy \nonumber\\
&\quad +C \int_\R \LC \int_\R e^{-|x-y|}(1+u_x^2)\,dx \RC \LB |v_y+u_{yy}w|(1+u_y^2) \right. \label{2.32}\\
&\quad \quad \quad \left. +|4u_y(v_y+u_{yy}w)+(1+u_y^2)w_y|(1+u_y^2)+|4u_y^2(v_y+u_{yy}w)|\RB \,dy \nonumber\\
&\leq  C \int_\R|v+u_xw|(1+u_x^2)^2\,dx+C\int_\R |v_x+u_{xx}w|(1+u_x^2)\,dx \nonumber\\
&\quad+C\int_\R |w|(1+u_x^2)^2\,dx+C\int_\R |4u_x(v_x+u_{xx}w)+(1+u_x^2)w_x|(1+u_x^2)\,dx \nonumber \\
& \leq C(I_2 + I_3 + I_1 + I_4). \nonumber
\end{align} 
 On the other hand, for the third term in the last equality of \eqref{2.31}, we notice that
 \begin{equation}\label{2.33}
 \begin{split}
3u_yv_yu&=3uu_y(v_y+u_{yy}w)-3uu_yu_{yy}w\\
&=3uu_y(v_y+u_{yy}w)-\frac{3}{2}\big(uu_y^2w\big)_y+\frac{3}{2}u^3_yw+\frac{3}{2}uu^2_yw_y,
\end{split}\end{equation}
and
 \begin{equation} \label{2.34}
 \begin{split}
\frac{3}{2}uu^2_yw_y & \leq \frac{3}{4} \LV u(1+u_y^2)u_yw_y \RV\\
& \leq \frac{3}{4}|u| \cdot \LV (1+u_y^2)u_yw_y+4u^2_y(v_y+u_{yy}w) \RV +3|u| \cdot \LV u^2_y(v_y+u_{yy}w) \RV\\
& \leq C|u|\cdot \LV (1+u_y^2)w_y+4u_y(v_y+u_{yy}w)\RV (1+u_y^2)+3|u| \cdot \LV u^2_y(v_y+u_{yy}w)\RV.
\end{split}
\end{equation}
Thus, by \eqref{2.33}, \eqref{2.34} and integration by parts, we bound the third term in the following way.
\begin{align}
&\quad\frac{1}{2}\LV \int_\R \LC\int_\R e^{-|x-y|}(3u_y v_y u+\frac{3}{2}u^2_y v+3u^2v)\,dy \RC(1+u_x^2)\,dx \RV \nonumber\\
& \leq C\int_\R \LC \int_\R e^{-|x-y|}(1+u_x^2)\,dx \RC \LCB 3|uu_y(v_y+u_{yy}w)|+3|u||u^2_y(v_y+u_{yy}w)| \right. \nonumber\\
&\quad +|u| \LV(1+u_y^2)w_y+4u_y(v_y+u_{yy}w) \RV (1+u_y^2)+\frac{3}{2}|u_y^2(v+u_yw)|\nonumber \\
&\quad \left. +3|u^2(v+u_yw)| +3|u^2u_yw|\RCB\,dy+C\int_\R\int_\R \LV \frac{\partial e^{-|x-y|}}{\partial y}\RV (1+u_x^2)\,dx \cdot |uu_y^2w|\,dy   \nonumber  \\
& \leq C\int_\R |v_x+u_{xx}w|(1+u_x^2)\,dx+C\int_\R |w|(1+u_x^2)^2\,dx   \label{2.35} \\
&\quad +C\int_\R |4u_x(v_x+u_{xx}w)+(1+u_x^2)w_x|(1+u_x^2)\,dx+C\int_\R|v+u_xw|(1+u_x^2)^2\,dx \nonumber\\
& \leq C (I_3 + I_1 + I_4 + I_2). \nonumber
\end{align}
Plugging \eqref{2.32} and \eqref{2.35} into \eqref{2.31}, it holds that
\begin{equation}\label{2.36}
{dI_3\over dt} = \frac{d}{dt}\int_\R |v_x+u_{xx}w|(1+u_x^2)\,dx \leq C(I_1 + I_2 + I_3 + I_4).
\end{equation}

{\bf 4.} To estimate the time derivative of $I_4$, differentiating \eqref{2.18} with respect to $x$ we get 
\begin{equation}\label{2.22}
w_{tt}+u^2w_{xx}=2u_x(v+u_x w)+2u(v_x+u_{xx}w).
\end{equation}
Using  \eqref{2.3}, \eqref{2.15}, \eqref{2.18}, \eqref{2.20} and \eqref{2.22}, it holds that
\begin{align*}
&\quad\LC 4(u_x+u_x^3)(v_x+u_{xx}w)+(1+u_x^2)^2 w_x\RC_t+\LC u^2\big(4(u_x+u_x^3)(v_x+u_{xx}w)+(1+u_x^2)^2 w_x\big)\RC_x\\
& = 4\LB (1+3u_x^2)(u_{xt}+u^2u_{xx})+2uu_x^2(1+u_x^2))\RB (v_x+u_{xx}w)+4\LB v_{xt}+u^2v_{xx}+u_{xx}(w_t+u^2w_x) \right. \\
&\quad \left.+w(u_{xxt}+u^2u_{xxx})\RB (u_x+u_x^3)+(1+u_x^2)^2(w_{xt}+u^2w_{xx})+w_x\LB \LC (1+u_{x}^2)^2\RC_t+\LC u^2(1+u_x^2)^2\RC_x\RB \\
& = 4\LB (1+3u_x^2)\LC -\frac{1}{2}uu^2_x+u^3-P_1-\partial_xP_2\RC +2uu_x^2(1+u_x^2))\RB (v_x+u_{xx}w) \\
&\quad +4(u_x+u_x^3)\LB -uu_xv_x -\frac{1}{2}u_x^2v-2uvu_{xx}+3u^2v-\frac{1}{2}\int_\R e^{-|x-y|}(3u_y v_y u+\frac{3}{2}u^2_y v+3u^2v)\,dy \right.\\
&\quad \left.-\frac{3}{4}\LC \int_x^{\infty}-\int^x_{-\infty}\RC e^{-|x-y|}u^2_yv_y\,dy+2uu_{xx}(v+u_xw)+w\LC 3u^2u_x-3uu_xu_{xx}-\partial_{x} P_1-P_2 \RC\RB \\
&\quad +(1+u_x^2)^2\LB 2u_x(v+u_xw)+2u(v_x+u_{xx}w)\RB +w_x(u^3_x+u_x)\LB 4u^3-4(P_1+\partial_x P_2)+2u\RB \\
& = \LB 6uu_x^2+2u+4(1+3u^2_x)\LC u^3-P_1-\partial_xP_2 \RC\RB(v_x+u_{xx}w)+2u_x(1+u_x^2)\LB v+u_x(1+u_x^2)w\RB \\
&\quad -3(u_x+u_x^3)\LC \int_x^{\infty}-\int^x_{-\infty}\RC e^{-|x-y|}u^2_yv_y\,dy+4(u_x+u_x^3)\LB 3u^2(v+u_xw)-w(\partial_{x} P_1+P_2) \right. \\
&\quad \left. -\frac{1}{2}\int_\R e^{-|x-y|}(3u_y v_y u+\frac{3}{2}u^2_y v+3u^2v)\,dy \RB +w_x(u^3_x+u_x)\LB 4u^3-4(P_1+\partial_x P_2)+2u \RB.
\end{align*}
For the second and third terms in the last equality, we have
\begin{equation}\label{2.38}
\begin{split}
&\quad 2u_x(1+u_x^2)\LB v+u_x(1+u_x^2)w\RB -3(u_x+u_x^3)\LC \int_x^{\infty}-\int^x_{-\infty}\RC e^{-|x-y|}u^2_yv_y\,dy\\
&=-u_x(1+u_x^2)\LC \int_x^{\infty}-\int^x_{-\infty}\RC \LB \LC e^{-|x-y|}\big(v+u_y(1+u_y^2)w\big)\RC_y+3e^{-|x-y|}u^2_yv_y\RB \,dy\\
&=-u_x(1+u_x^2)\LC\int_x^{\infty}-\int^x_{-\infty}\RC e^{-|x-y|}\LC v_y+u_{yy}w+3u^2_y(v_y+u_{yy}w)+u_y(1+u_y^2)w_y\RC\,dy\\
&\quad+u_x(1+u_x^2)\int_\R e^{-|x-y|}\LC v+u_y(1+u_y^2)w\RC \,dy\\
&=-u_x(1+u_x^2)\LC \int_x^{\infty}-\int^x_{-\infty}\RC e^{-|x-y|}\LB v_y+u_{yy}w+4u^2_y(v_y+u_{yy}w)+u_y(1+u_y^2)w_y \right.\\
&\qquad\qquad\qquad  \left. -u^2_y(v_y+u_{yy}w)\RB\,dy+u_x(1+u_x^2)\int_\R e^{-|x-y|}\big(v+u_yw+u_y^3w\big)\,dy.\\
\end{split}
\end{equation}
While for the last term,
\begin{equation}\label{2.39}
\begin{split}
&\quad w_x(u^3_x+u_x)(4u^3-4(P_1+\partial_x P_2)+2u)\\
&= \LC 4u^3-4(P_1+\partial_x P_2)+2u \RC  \LB w_x(u^3_x+u_x)+4u^2_x(v_x+u_{xx}w)-4u^2_x(v_x+u_{xx}w) \RB.
\end{split}
\end{equation}
Finally, the term $-2(u_x+u_x^3)\int^{\infty}_{-\infty}e^{-|x-y|}(3u_y v_y u+\frac{3}{2}u^2_y v+3u^2v)(y)\,dy$ can be estimated in a similar way as is done to obtain \eqref{2.35}.
We thus conclude from all of the above that
\begin{equation}\label{2.40}
{dI_4 \over dt} = \frac{d}{dt}\int_\R \LV 4(u_x+u_x^3)(v_x+u_{xx}w)+(1+u_x^2)^2 w_x \RV\,dx \leq C(I_1 + I_2 + I_3 + I_4).
\end{equation}

Combining the inequalities \eqref{2.24}, \eqref{2.30}, \eqref{2.36} and \eqref{2.40} together, we obtain the desired inequality \eqref{est on w and tilde v}, and hence \eqref{2.19}.
\end{proof}

%

\section{Generic regularity for the Novikov equation: Proof of Theorem  \ref{thm_generic}.}\label{sec_generic regularity}
\setcounter{equation}{0}
In this section, we use the Thom's transversality Theorem to prove generic regularity result given in Theorem  \ref{thm_generic}. 

\subsection{The semi-linear system on new coordinates}
As a start, we first briefly review the semi-linear system introduced in \cite{CCL}, which will be used later. Please find detail calculations and derivations in \cite{CCL}.

Following the idea in \cite{CCL}, we introduce new coordinates $(t,Y)$, such that
\begin{equation*}
Y\equiv Y(t,x):=\int_0^{x^c (0;t,x)} (1+u^2_x(0,x'))^2\,dx',
\end{equation*}
where $a\mapsto x^c(a;t,x)$ denotes a characteristic passing through the point $(t,x)$. Here
the equation of characteristic is $$\frac{dx(t)}{dt}=u^2(t,x(t)).$$

In fact, $Y=Y(t,x)$ is a characteristic coordinate, which satisfies $Y_t+u^2Y_x=0$ for any $(t,x)\in\mathbb{R}^+\times\mathbb{R}$. 
We also denote
\begin{equation}\label{aluxi}
\alpha=2\arctan u_x \quad\text{ and } \quad\xi=\frac{(1+u_x^2)^2}{Y_x},
\end{equation}
then one can derive a semi-linear system
\begin{equation}\label{4.1}
\begin{cases}
& u_t(t,Y)=-\partial_x P_1-P_2,\\
& \alpha_t(t,Y)=-u\sin^2\frac{\alpha}{2}+2u^3\cos^2\frac{\alpha}{2}-2\cos^2\frac{\alpha}{2}(P_1+\partial_x P_2),\\
& \xi_t(t,Y)=\xi[(2u^3+u)-2(P_1+\partial_x P_2)]\sin \alpha.
\end{cases}
\end{equation}
Here the initial conditions given as 
\begin{equation}\label{4.3}
\begin{cases}
& u(0,Y)=u_0(x(0,Y)),\\
& \alpha(0,Y)=2\arctan u_{0,x}(x(0,Y)),\\
& \xi(0,Y)=1,
\end{cases}
\end{equation}
where
\begin{equation}\label{4.4}
\begin{split}
&\displaystyle P_1(Y)=\frac{1}{2}\int_{-\infty}^\infty e^{-|\int_{\bar{Y}}^Y( \xi\cos^4\frac{\alpha}{2})(t,\tilde{Y})\,d\tilde{Y}|}\LC \frac{3}{8}u\sin^2 \alpha+u^3\cos^4\frac{\alpha}{2}\RC \xi(t,\bar{Y})\,d\bar{Y},\\
&\displaystyle \partial_x P_1(Y)=\frac{1}{2}\LC \int^\infty_Y-\int_{-\infty}^Y \RC e^{-|\int_{\bar{Y}}^Y( \xi\cos^4\frac{\alpha}{2})(t,\tilde{Y})\,d\tilde{Y}|}\LC \frac{3}{8}u\sin^2\alpha+u^3\cos^4\frac{\alpha}{2}\RC \xi(t,\bar{Y})\,d\bar{Y},\\
&\displaystyle P_2(Y)=\frac{1}{8}\int_{-\infty}^\infty e^{-|\int_{\bar{Y}}^Y( \xi\cos^4\frac{\alpha}{2})(t,\tilde{Y})\,d\tilde{Y}|}\LC \xi\sin \alpha\sin^2\frac{\alpha}{2}\RC (t,\bar{Y})\,d\bar{Y},\\
&\displaystyle \partial_x P_2(Y)=\frac{1}{8}\LC \int^\infty_Y-\int_{-\infty}^Y \RC e^{-|\int_{\bar{Y}}^Y( \xi\cos^4\frac{\alpha}{2})(t,\tilde{Y})\,d\tilde{Y}|}\big(\xi\sin \alpha\sin^2\frac{\alpha}{2}\big)(t,\bar{Y})\,d\bar{Y}.\\
\end{split}
\end{equation}

By  expressing the solution $u(t,Y)$ in terms of the original variables $(t,x)$, one obtains a weak solution of the Cauchy problem \eqref{2.1} as stated in Theorem \ref{thm_exist}. Furthermore, this solution is proved to be the unique conservative solution in \cite{CCL}. Indeed, the following result is also proved in \cite{CCL}.

\begin{Lemma}\label{Lemma 4.1}
Let $(x,u,\alpha,\xi)(t,Y)$ be the solution to the system \eqref{4.1}--\eqref{4.3} with $\xi>0$. Then the set of points
\begin{equation}\label{4.5}
\mathrm{Graph}(u)=\{(t,x(t,Y),u(t,Y)):\quad (t,Y)\in \mathbb{R}^+\times\mathbb{R}\}
\end{equation}
is the graph of a unique conservative solution to the Novikov equation \eqref{2.1}.
\end{Lemma}

\subsection{Families of perturbed solutions}
To begin with, we state the following technical lemma (cf. \cite{LZ}) which will be used  later.
\begin{Lemma}\label{Lemma 4.2}
Consider an ODE system
\[\frac{d}{dt}\mathbf{u}^\varepsilon=f(\mathbf{u}^\varepsilon),\quad \mathbf{u}^\varepsilon(0)=\mathbf{u}_0+\varepsilon_1\mathbf{v}_1+\cdots+\varepsilon_m\mathbf{v}_m,\]
where $\mathbf{u}^\varepsilon(t): \mathbb{R}\to \mathbb{R}^n$, $f$ is a Lipschitz continuous function. The system is well-posed in $[0,t^*)$. Let the matrix 
\[D_\varepsilon\mathbf{u}^\varepsilon(0)=(\mathbf{v}_1,\mathbf{v}_2,\cdots,\mathbf{v}_m)\in\mathbb{R}^{n\times m},\]
have rank  $\text{\bf rank } (D_\varepsilon\mathbf{u}^\varepsilon(0))=k$. Then for any $t\in [0,t^*)$, $$\text{\bf rank }(D_\varepsilon\mathbf{u}^\varepsilon(t))=k.$$

\end{Lemma}

Now, for a fixed solution of \eqref{4.1}, we are going to construct several families of perturbed solutions.
\begin{Lemma}\label{Lemma 4.3}
Let $(u,\alpha,\xi)$ be a smooth solution of the semi-linear  system \eqref{4.1} and let a point $(t_0,Y_0)\in \mathbb{R}^+\times\mathbb{R}$ be given.

{\rm (1)} If $(\alpha,\alpha_Y,\alpha_{YY})(t_0,Y_0)=(\pi,0,0)$, then there exists a 3--parameter family of smooth solutions $(u^\vartheta,\alpha^\vartheta,\xi^\vartheta)$ of \eqref{4.1}, depending smoothy on $\vartheta\in\mathbb{R}^3$, such that the following holds.

{\rm (i)} when $\vartheta=0\in\mathbb{R}^3$, one recovers the original solution, namely $(u^0,\alpha^0,\xi^0)=(u,\alpha,\xi)$.

{\rm (ii)} At the point $(t_0,Y_0)$, when $\vartheta=0$ one has
\begin{equation}\label{4.6}
\text{\bf rank } D_\vartheta(\alpha^\vartheta,\alpha_Y^\vartheta,\alpha_{YY}^\vartheta)=3.
\end{equation}

{\rm (2)} If $(\alpha,\alpha_Y,\alpha_{t})(t_0,Y_0)=(\pi,0,0)$, then there exists a 3--parameter family of smooth solutions $(u^\vartheta,\alpha^\vartheta,\xi^\vartheta)$ of \eqref{4.1}, depending smoothy on $\vartheta\in\mathbb{R}^3$, satisfying {\rm (i)--(ii)} as above, with \eqref{4.6} replaced by 
\begin{equation}\label{4.7}
\text{\bf rank } D_\vartheta(\alpha^\vartheta,\alpha_Y^\vartheta,\alpha_{t}^\vartheta)=3.
\end{equation}
\end{Lemma}

\begin{proof}
Let $(u,\alpha,\xi)$ be a smooth solution of the semi-linear  system \eqref{4.1}. Given the point $(t_0,Y_0)$. Taking derivatives to the equations of $\alpha$ and $\xi$ in \eqref{4.1}, we have
\begin{equation}\label{4.8}
\begin{split}
\frac{\partial }{\partial t}\alpha_Y(t,Y) & = u_Y(6u^2\cos^2\frac{\alpha}{2}-\sin^2\frac{\alpha}{2})-\alpha_Y\sin \alpha(\frac{1}{2}u+u^3-P_1-\partial_x P_2)\\
& \quad -2\cos^2\frac{\alpha}{2}\xi\big(\cos^4\frac{\alpha}{2}(\partial_x P_1+ P_2)-\frac{1}{4}\sin^2\frac{\alpha}{2}\sin \alpha\big),
\end{split}\end{equation}

\begin{equation}\label{4.9}
\begin{split}
\frac{\partial }{\partial t}\alpha_{YY}(t,Y) & = u_{YY}(6u^2\cos^2\frac{\alpha}{2}-\sin^2\frac{\alpha}{2})-\alpha_Yu_Y\sin \alpha(1+6u^2)\\
&\quad -(\alpha_{YY}\sin \alpha+\alpha_Y^2\cos \alpha)(\frac{1}{2}u+u^3-P_1-\partial_x P_2)\\
& \quad +\alpha_Y\xi\sin \alpha\big(4\cos^4\frac{\alpha}{2}(\partial_x P_1+P_2)-\frac{1}{2}\sin^2\frac{\alpha}{2}\sin \alpha\big)\\
& \quad -2\cos^2\frac{\alpha}{2}\xi_Y\big(\cos^4\frac{\alpha}{2}(\partial_x P_1+ P_2)-\frac{1}{4}\sin^2\frac{\alpha}{2}\sin \alpha\big)\\
& \quad -2\cos^6\frac{\alpha}{2}\xi^2\big(\cos^4\frac{\alpha}{2}( P_1+ \partial_xP_2-u^3)-\frac{3}{8}u\sin^2 \alpha\big)\\
& \quad +\frac{1}{4}\cos^2\frac{\alpha}{2}\alpha_Y\xi\big(\sin^2 \alpha+2\sin^2\frac{\alpha}{2}\cos \alpha\big),
\end{split}\end{equation}

\begin{equation}\label{4.10}
\begin{split}
\frac{\partial }{\partial t}\alpha_{t}(t,Y) & = (\partial_x P_1+ P_2)(\sin^2\frac{\alpha}{2}-6u^2\cos^2\frac{\alpha}{2})-2\cos^2\frac{\alpha}{2}(\partial_t P_1+ \partial_t \partial_x P_2)\\
& \quad -\sin \alpha\big(-u\sin^2\frac{\alpha}{2}+2u^3\cos^2\frac{\alpha}{2}-2\cos^2\frac{\alpha}{2}(P_1+\partial_x P_2)\big)(\frac{1}{2}u+u^3-P_1-\partial_x P_2),
\end{split}\end{equation}

\begin{equation}\label{4.11}
\begin{split}
\frac{\partial }{\partial t}\xi_Y(t,Y) & = (\xi_Y\sin \alpha+\xi \alpha_Y\cos \alpha)\big(2u^3+u-2(P_1+\partial_x P_2)\big)\\
& \quad +\xi\sin \alpha\big(6u^2 u_Y+u_Y-2\xi\cos^4\frac{\alpha}{2}(\partial_x P_1+P_2)+\frac{1}{2}\xi\sin \alpha\sin^2\frac{\alpha}{2}\big),
\end{split}\end{equation}
with 
\begin{equation*}
u_Y(t,Y)=\frac{1}{2}\xi\sin \alpha\cos^2\frac{\alpha}{2},\quad u_{YY}=\frac{1}{2}\xi_Y\sin \alpha\cos^2\frac{\alpha}{2}+\frac{1}{2}\xi \alpha_Y\cos \alpha\cos^2\frac{\alpha}{2}-\frac{1}{4}\xi \alpha_Y\sin^2 \alpha.
\end{equation*}
Combining \eqref{4.1}, \eqref{4.8}, \eqref{4.9} and \eqref{4.11}, we obtain a complete system.
Now, we construct a family of solutions $(\bar{u}^\vartheta,\bar{\alpha}^\vartheta,\bar{\xi}^\vartheta)$ to the complete system \eqref{4.1}, \eqref{4.8}, \eqref{4.9} and \eqref{4.11} with initial data being the perturbations of \eqref{4.3}
\begin{equation}\label{4.12}
\begin{cases}
&\displaystyle \bar{u}^\vartheta(0,Y)=u(0,Y)+\sum_{i=1,2,3}\vartheta_i U_i(Y),\\
&\displaystyle \bar{\alpha}^\vartheta(0,Y)=\alpha(0,Y)+\sum_{i=1,2,3}\vartheta_i V_i(Y),\\
&\displaystyle \bar{\xi}^\vartheta(0,Y)=\xi(0,Y)+\sum_{i=1,2,3}\vartheta_i \zeta_i(Y),
\end{cases}\end{equation}
for some suitable functions $U_i, V_i, \zeta_i\in \mathcal{C}^\infty_0(\mathbb{R})$. That is, consider the system\begin{equation*}
\frac{\partial}{\partial t}\left(                   
\begin{array}{c}  
    \bar{u}^\vartheta\\  
 \bar{\alpha}^\vartheta\\ 
   \bar{\xi}^\vartheta\\ 
 \bar{\alpha}_Y^\vartheta\\
  \displaystyle  \bar{\alpha}_{YY}^\vartheta\\
 \bar{\xi}_Y^\vartheta
  \end{array}
\right)     =\left(                   
\begin{array}{c}  
    f_1^\vartheta\\  
f_2^\vartheta\\ 
   f_3^\vartheta\\ 
    f_4^\vartheta\\
 f_5^\vartheta\\
  f_6^\vartheta
  \end{array}
\right),          
\end{equation*}
with the initial data \eqref{4.12}, where $f_1^\vartheta,\cdots, f_6^\vartheta$ are the perturbations of the right-hand side of \eqref{4.1}, \eqref{4.8}, \eqref{4.9} and \eqref{4.11}. Thanks to Lemma \ref{Lemma 4.2}, it remains  to prove the Lipschitz continuity of $f_1^\vartheta,\cdots, f_6^\vartheta$. Observe that $(u,\alpha,\xi)$ is smooth, so we just need to consider the Lipschitz continuity of $P_j$ and $\partial_x P_j$ for $j=1,2$. From the definition of the source term $P_j$ at \eqref{4.4}, it follows
\begin{equation*}
\begin{split}
\left|\frac{\partial P_1}{\partial u}(t,Y) \right| & = \frac{1}{2}|\int_{-\infty}^\infty e^{-|\int_{\bar{Y}}^Y( \xi\cos^4\frac{\alpha}{2})(t,\tilde{Y})\,d\tilde{Y}|}\big(\frac{3}{8}\sin^2 \alpha+3u^2\cos^4\frac{\alpha}{2}\big)\xi(t,\bar{Y})\,d\bar{Y}|\\
& \leq \frac{1}{2}\int_{-\infty}^\infty e^{-|x-\bar{x}|}(\frac{3}{2}u_x^2+3u^2)(t,\bar{x})\,d\bar{x}\leq \frac{3}{2}\mathcal{E}(0).
\end{split}\end{equation*}
Similarly, by \eqref{2.8} and \eqref{4.4}, we can obtain
\begin{equation*}
\begin{split}
\left|\frac{\partial P_1}{\partial \alpha}(t,Y) \right| & \leq \frac{1}{2}\int_{-\infty}^\infty e^{-|\int_{\bar{Y}}^Y( \xi\cos^4\frac{\alpha}{2})(t,\tilde{Y})\,d\tilde{Y}|}|\int_{\bar{Y}}^Y (\xi\cos^2\frac{\alpha}{2}\sin \alpha)(t,\tilde{Y})\,d\tilde{Y}|\\
& \quad \cdot|\frac{3}{8}u\sin^2 \alpha+u^3\cos^4\frac{\alpha}{2}\big)|\xi(t,\bar{Y})\,d\bar{Y}\\
& \quad +\frac{1}{2}\int_{-\infty}^\infty e^{-|\int_{\bar{Y}}^Y( \xi\cos^4\frac{\alpha}{2})(t,\tilde{Y})\,d\tilde{Y}|}|\frac{3}{4}u\sin \alpha\cos \alpha-u^3\cos^2\frac{\alpha}{2}\sin \alpha|\xi(t,\bar{Y})\,d\bar{Y}\\
& \leq \frac{1}{2}\int_{-\infty}^\infty e^{-|x-\bar{x}|}|2\int_{\bar{x}}^x u_x(t,\tilde{x})\,d\tilde{x}||(\frac{3}{2}u u_x^2+u^3)(t,\bar{x})|\,d\bar{x}\\
& \quad +\frac{1}{2}\int_{-\infty}^\infty e^{-|x-\bar{x}|}|(\frac{3}{2}u u_x-\frac{3}{2}u u_x^3-2u^3u_x)(t,\bar{x})|\,d\bar{x}\\
& \leq \int_{-\infty}^\infty e^{-|x-\bar{x}|}|x-\bar{x}|^\frac{1}{2}\mathcal{E}(0)^\frac{1}{2}|(\frac{3}{2}u u_x^2+u^3)(t,\bar{x})|\,d\bar{x}+\frac{3}{4}\|u\|_{L^2}\|u_x\|_{L^2}\\
& \quad +\frac{3}{4}\|u\|_{L^\infty}\|u_x\|_{L^3}+\|u\|_{L^\infty}^2\|u\|_{L^2}\|u_x\|_{L^2}\\
& \leq \frac{3\sqrt{2}}{2}\mathcal{E}(0)^{\frac{3}{2}}+\frac{3}{4}\mathcal{E}(0)+\frac{3}{4}\mathcal{E}(0)^\frac{1}{2}K^{\frac{1}{3}}+2\mathcal{E}(0)^2,
\end{split}\end{equation*}
where $K$ is defined in \eqref{2.11}.
In a similar fashion as the above two estimates and the fact that $\xi$ is bounded, it is easy to verify the boundedness of $|\frac{\partial P_2}{\partial u}|$, $|\frac{\partial P_2}{\partial \alpha}|,\frac{\partial P_j}{\partial \xi}|, |\partial_u\partial_x P_j|, |\partial_\alpha\partial_x P_j|, |\partial_\xi\partial_x P_j|$ for $j=1,2$. This completes the proof of the Lipschitz continuity of $f_1^\vartheta,\cdots, f_6^\vartheta$.

Thus, by choosing suitable perturbations $V_i$, $i=1,2,3$, at the point $(t_0,Y_0)$ and $\vartheta=0$, using Lemma \ref{Lemma 4.2}, we can get 
\begin{equation*}
\text{ \bf rank } D_\vartheta\left(                   
\begin{array}{c}  
    \alpha\\  
 \alpha_Y\\ 
     \alpha_{YY}\\
  \end{array}
\right) =3.
\end{equation*}

On the other hand, \eqref{4.1}, \eqref{4.8} and \eqref{4.10} form a complete system. Similar to the proof of \eqref{4.6}, we can choose suitable perturbations $V_i$, $i=1,2,3$, such that, at the point $(t_0,Y_0)$ and $\vartheta=0$, we have
 \begin{equation*}
\text{ \bf rank } D_\vartheta\left(                   
\begin{array}{c}  
    \alpha\\  
 \alpha_Y\\ 
     \alpha_{t}\\
  \end{array}
\right) =3.
\end{equation*}
This completes the proof  of Lemma \ref{Lemma 4.3}.
\end{proof}

\subsection{Proof of Theorem \ref{thm_generic}}\label{sub4.2}
Now, we will use Lemma \ref{Lemma 4.3} together with Transversality argument to  study the smooth solutions to the semi-linear system \eqref{4.1}, and hence determine the generic structure of the level sets $\{(t,Y); \alpha(t,Y)=\pi\}$. The proof of the following lemma is similar to \cite{LZ}, and we omit it here for brevity. 
\begin{Lemma}\label{Lemma 4.4}
Let a compact domain $$\Omega:=\{(t,Y); ~0\leq t\leq T,|Y|\leq M\},$$
and define $S$ to be the family of all $\mathcal{C}^2$ solutions  $(u,\alpha,\xi)$ to the semi-linear system \eqref{4.1}, with $\xi>0$ for all $(t,Y)\in\mathbb{R}^+\times\mathbb{R}$. Moreover, define $S'\subset S$ to be the subfamily of all solutions $(u,\alpha,\xi)$, such that for $(t,Y)\in\Omega$, none of the following values is attained:
\begin{equation}\label{4.13}
(\alpha,\alpha_Y,\alpha_{YY})=(\pi,0,0),\quad (\alpha,\alpha_Y,\alpha_t)=(\pi,0,0).
\end{equation}
Then $S'$ is a relatively open and dense subset of $S$, in the topology induced by $\mathcal{C}^2(\Omega)$.
\end{Lemma}

With the help of Lemma \ref{Lemma 4.4}, we can now prove Theorem \ref{thm_generic} for the generic regularity of conservative solutions to the Novikov equation \eqref{2.1}.

\begin{proof}[Proof of Theorem \ref{thm_generic}]
First, we denote $$\mathcal{N}:=\mathcal{C}^3(\mathbb{R})\cap H^1(\mathbb{R})\cap W^{1,4}(\mathbb{R}),$$ with norm $$\|u_0\|_{\mathcal{N}}:=\|u_0\|_{\mathcal{C}^3}+\|u_0\|_{H^1}+\|u_0\|_{W^{1,4}}.$$
Given initial data $\hat{u}_0\in\mathcal{N}$ and denote the open ball 
\[B_\delta:=\{u_0\in\mathcal{N};~\|u_0-\hat{u}_0\|_{\mathcal{N}}<\delta\}.\]
Now, we will prove our theorem by showing that, for any $\hat{u}_0\in\mathcal{N}$, there exists a radius $\delta>0$ and an open dense subset $\tilde{\mathcal{D}}\subset B_{\delta}$, with the following property: for every initial data $u_0\in\tilde{\mathcal{D}}$, the conservative solution $u=u(t,x)$ of \eqref{2.1} is twice continuously differentiable in the complement of finitely many characteristic curves within the domain $[0,T]\times\mathbb{R}$.

{\bf 1.} Since $u_0\in \mathcal{N}$, by the definition of the space $\mathcal{N}$, we have
$$u_0(x)\to 0\quad \text{ and }\quad u_{0,x}(x)\to 0, \quad\text{ as }|x|\to\infty.$$
Thus, we can choose $r>0$ sufficiently large, such that $u_0(x)$ and $u_{0,x}(t,x)$ being uniformly bounded for all $|x|\geq r$. On a domain of the form $\{(t,x);~ t\in[0,T], |x|\geq r+T\|u\|_{L^\infty}^2\}$, a standard comparison argument yields that the partial derivative $u_x$ remains uniformly bounded. This implies that the singularities of $u(t,x)$ in the set $[0,T]\times\mathbb{R}$ only appear on the compact set $\mathcal{M}:=[0,T]\times [-r-T\|u\|_{L^\infty}^2, r+T\|u\|_{L^\infty}^2].$

Next, for any $u_0\in B_{\delta}$, let $\Lambda$ be the map $(t,Y)\mapsto \Lambda(t,Y):= (t,x(t,Y))$, and let $\Omega$ be a domain as in Lemma \ref{Lemma 4.4}. Choosing $M$ large enough and by possibly shrinking the radius $\delta$, we can obtain the inclusion $\mathcal{M}\subset \Lambda(\Omega)$.

Now, the subset $\tilde{\mathcal{D}}\subset B_\delta$ is defined as follows. $u_0\in \tilde{\mathcal{D}}$ if $u_0\in B_\delta$ and for the corresponding solution $(u,\alpha,\xi)$ of \eqref{4.1}, the values \eqref{4.13} are never attained for any $(t,Y)$ such that $(t,x(t,Y))\in\mathcal{M}.$ In the next two steps, we examine $\tilde{\mathcal{D}}$ is an open dense set.

{\bf 2.} To prove that the set $\tilde{\mathcal{D}}$ is open in the topology of $\mathcal{C}^3$, we choose a sequence of initial data $(u_0^\nu)_{\nu\geq 1}$ such that the sequence converges to $u_0$, with $u_0^\nu\notin \tilde{\mathcal{D}}$. By the definition of $\tilde{\mathcal{D}}$, there exist points $(t^\nu,Y^\nu)$ at which the corresponding solutions $(u^\nu,\alpha^\nu,\xi^\nu)$ satisfy
\[(\alpha^\nu,\alpha_Y^\nu,\alpha_{YY}^\nu)(t^\nu,Y^\nu)=(\pi,0,0), \quad (t^\nu,x^\nu(t^\nu,Y^\nu))\in \mathcal{M}, \]
for all $\nu\geq 1.$ Observe that the domain $\mathcal{M}$ is compact, thus we can choose a subsequence, denote still by $(t^\nu,Y^\nu)$, such that $ (t^\nu,Y^\nu)\to (\bar{t},\bar{Y})$ for some point $(\bar{t},\bar{Y})$. By continuity, 
\[(\alpha,\alpha_Y,\alpha_{YY})(\bar{t},\bar{Y})=(\pi,0,0), \quad (t,x(\bar{t},\bar{Y}))\in\mathcal{M},\]
which implies $u_0\notin \tilde{\mathcal{D}}$. Repeating the same procedure when $(\alpha,\alpha_Y,\alpha_T)=(\pi,0,0)$, it follows that $\tilde{\mathcal{D}}$ is open.

{\bf 3.} Now, we will prove the set $\tilde{\mathcal{D}}$ is dense in $B_\delta$. Given $u_0\in B_\delta$, by a small perturbation, we can assume that $u_0\in\mathcal{C}^\infty$.

From Lemma \ref{Lemma 4.4}, we can construct a sequence of solutions $(u^\nu,\alpha^\nu,\xi^\nu)$ of \eqref{4.1}, such that, for every $\nu\geq 1, (t,Y)\in\Omega$, the values in \eqref{4.13} are never attained, and the $\mathcal{C}^k, k\geq 1$ norm satisfies
\begin{equation*}
\lim_{\nu\to\infty}\|(u^\nu-u,\alpha^\nu-\alpha,\xi^\nu-\xi,x^\nu-x)\|_{\mathcal{C}^k(\Gamma)}=0,
\end{equation*}
for every bounded set $\Gamma\subset [0,T]\times\mathbb{R}$. Thus, for $t=0$, the corresponding sequence of initial values satisfies
\begin{equation}\label{4.14}
\lim_{\nu\to\infty}\|u^\nu_0-u_0\|_{\mathcal{C}^k(I)}=0,
\end{equation}
for every bounded set $I\subset \mathbb{R}$.

Consider a cutoff function $\psi(x)\in\mathcal{C}_0^\infty$, such that
\begin{equation*}\begin{split}
&\psi(x)=1,\qquad \text{ if } |x|\leq l,\\
&\psi(x)=0,\qquad  \text{ if }  |x|\geq l+1,
\end{split}\end{equation*}
where $l\gg r+T\|u\|_{L^\infty}^2$ is large enough. Then for every $\nu\geq 1$, define the following initial data
\[\tilde{u}_0^\nu:=\psi u_0^\nu+(1-\psi)u_0,\]
which together with \eqref{4.14} implies
\begin{equation*}
\lim_{\nu\to\infty}\|\tilde{u}_0^\nu-u_0\|_{\mathcal{N}}=0.
\end{equation*}
Further enlarge $l$ such that for every $(t,x)\in \mathcal{M}$,
\[\tilde{u}^\nu(t,x)=u^\nu(t,x).\]
Notice that $\tilde{u}^\nu(t,x)$ remains $\mathcal{C}^2$ in the outer domain $\{(t,x); ~t\in [0,T ], |x|\geq r+T\|u\|_{L^\infty}^2\}$. Thus, we have proved for every $\nu\geq 1$ sufficiently large, $\tilde{u}_0^\nu\in\tilde{D}$, which means that $\tilde{D}$ is dense in $B_\delta$.

{\bf 4.} Finally, we will show that for every initial data $u_0\in\tilde{D}$, the corresponding solution $u(t,x)$ of \eqref{2.1} is piecewise $\mathcal{C}^2$ on the domain $[0,T]\times\mathbb{R}$. Indeed, we know that $u$ is $\mathcal{C}^2$ in the outer domain $\{(t,x); ~t\in [0,T ], |x|\geq r+T\|u\|_{L^\infty}^2\}$ by the previous argument. So we need to study the singularity of $u$ in the inner domain $\mathcal{M}$. 

Recall from step  1, every point in $\mathcal{M}$ is contained in the image of the domain $\Omega$. Thus, for every point $(t_0,Y_0)\in\Omega$, we have two cases.

{\it Case I.} $\alpha(t_0,Y_0)\neq \pi$. From the coordinate change $x_Y=\xi\cos^4 \frac{\alpha}{2}, t=t$, we know that the map $(t,Y)\mapsto (t,x)$ is locally invertible in a neighborhood of $(t_0,Y_0)$. Therefore, the function $u$ is $\mathcal{C}^2$ in a neighborhood of $(t_0,x(t_0,Y_0))$.

{\it Case II.} $\alpha(t_0,Y_0)=\pi$. Since $u_0\in\tilde{D}$, so by the definition of $\tilde{D}$, we have $\alpha_t(t_0,Y_0)\neq 0$ or $\alpha_Y(t_0,Y_0)\neq 0$.

{\bf 5.} By continuity, there exists an $\eta>0,$ such that the values in \eqref{4.13} are never attained in the open neighborhood $$\Omega' :=\{(t,Y); ~0\leq t\leq T,|Y|\leq M+\eta\}.$$
Applying the implicit function theorem, we derive that the set \[S^\alpha:=\{(t,Y)\in\Omega'; ~\alpha(t,Y)=\pi\}\] is a one-dimensional embedded manifold of class $\mathcal{C}^2$. 

Now, we claim that the number of connected components of $S^\alpha$ that intersect the compact set $\Omega$ is finite. Assume, by contradiction, that $P_1, P_2,\cdots$ is a sequence of points in $S^\alpha\cap \Omega$ belonging to distinct components. Thus, we can choose a subsequence $P_i$, such that $P_i\to \bar{P}$ for some $\bar{P}\in S^\alpha\cap\Omega$.  By assumption, $(\alpha_t,\alpha_Y)(\bar{P})\neq (0,0)$. Hence, by the implicit function theorem, there is a neighborhood $\mathcal{U}$ of $\bar{P}$ such that $\gamma:=S^\alpha\cap\mathcal{U}$ is a connected $\mathcal{C}^2$ curve. Thus, $P_i \in \gamma$ for $i$ large enough, providing a contradiction.

{\bf 6.} To complete the proof, we need to study in more detail the image of the singular set $S^\alpha$, since the set of points $(t,x)$ where $u$ is singular coincides with the image of the set $S^\alpha$ under the $C^2$ map $(t,Y)\mapsto \Lambda(t,Y)= (t,x(t,Y))$. 

By the argument in step 5, inside the compact set $\Omega$, there are only finitely many points  where $\alpha=\pi, \alpha_Y=0, \alpha_t\neq 0$, say $P_i=(t_i,Y_i), i=1,\cdots, m$, and also where $\alpha=\pi, \alpha_t=0,\alpha_Y\neq 0$, say $Q_\jmath=(t_\jmath',Y_\jmath'), \jmath=1,\cdots, n$. 

By the analysis in step 5, the set $S^\alpha \verb|\| \{P_1,\cdots,P_m,Q_1,\cdots,Q_n\}$ has finitely many connected components which intersect $\Omega$. Consider any one of these components. This is a connected curve, say $\gamma_j$, such that $\alpha=\pi, \alpha_Y\neq 0$ for any $(t,Y)\in \gamma_j$. Thus, this curve can be expressed in the form $$\gamma_j=\{(t,Y):~Y=\phi_j(t), a_j<t<b_j\},$$ for a suitable function $\phi_j$. 

At this stage, we claim that the image $\Lambda(\gamma_j)$ is a $\mathcal{C}^2$ curve in the $t$--$x$ plane. Indeed, it suffices to show that, on the open interval $(a_j,b_j)$, the differential of the map $t\mapsto (t,x(t,\phi_j(t)))$ does not vanish. This is true, because 
\[\frac{d}{dt}x(t,\phi_j(t))=1+x_Y\phi'_j=1>0,\]
since $x_Y=\xi\cos^4\frac{\alpha}{2}=0 $ when $\alpha=\pi$. Hence, the singular set $\Lambda(S^\alpha)$ is the union of the finitely many points $p_i=\Lambda(P_i), i=1,\cdots, m,$ $ q_\jmath=\Lambda(Q_\jmath), \jmath=1,\cdots,n,$ together with finitely many $\mathcal{C}^2$--curve $\Lambda(\gamma_j)$.  This completes the proof of Theorem \ref{thm_generic}.
\end{proof}

\subsection{One-parameter families of solutions}
In this subsection, we study families of conservative solutions $u^\theta=u(t,x,\theta)$ of \eqref{2.1} with initial data $u(0,x,\theta)=u_0(x,\theta)$, depending smoothly on an additional parameter $\theta\in[0,1]$, 
More precisely, these paths of initial data will lie in the space
\[\mathcal{X}:=\mathcal{C}^3\big(\mathbb{R}\times[0,1]\big)\cap L^\infty\big([0,1];H^1(\mathbb{R})\cap W^{1,4}(\mathbb{R})\big).\]
Now, we have the following generic regularity for one-parameter family of solution. Roughly speaking, for a one-parameter family of initial data $\theta\mapsto \hat{u}_0^\theta$, with $\theta\in[0,1]$, it can be uniformly approximated by a second path of initial data $\theta\mapsto u_0^\theta$, such that the corresponding solutions $u^\theta=u^\theta(t,x)$ of \eqref{2.1} are piecewise smooth in the domain $[0,T]\times\mathbb{R}$. The proof is similar to \cite{BC}, and hence we omit it here for brevity. Note that this is the step where
Thom's Transversality Theorem is used. We refer the reader to \cite{BC} for the Thom's Transversality Theorem.

\begin{Theorem}\label{thm_para}
Let $T>0$ be given, then for any one-parameter family of initial data $\hat{u}_0^\theta\in\mathcal{X}$ and any $\varepsilon>0$, there exists a perturbed family $(x,\theta)\mapsto u_0(x,\theta)=:u^\theta_0(x)$ such that \[\|u_0^\theta-\hat{u}_0^\theta\|_{\mathcal{X}}<\varepsilon,\]
and moreover, for all except at most finitely many $\theta\in[0,1]$, the conservative solution $u^\theta=u(t,x,\theta)$ of \eqref{2.1} is smooth in the complement of finitely many points and finitely many $\mathcal{C}^2$ curves in the domain $[0,T]\times\mathbb{R}$.
\end{Theorem}

In accordance with the previous argument, we give the following definitions.
\begin{Definition}\label{Definition 4.1}
We say that a solution $u=u(t,x)$ of \eqref{2.1} has generic singularities for $t\in[0,T]$ if it admits a representation of the form \eqref{4.5}, where 

{\rm (i)} the functions $(x,u,\alpha,\xi)(t,Y)$ are $\mathcal{C}^\infty$,

{\rm (ii)} for $t\in[0,T]$, the generic condition 
\begin{equation}\label{generic_con}
\alpha=\pi, \alpha_Y=0 \Longrightarrow\alpha_t\neq 0,\alpha_{YY}\neq 0,
\end{equation}
holds.
\end{Definition}

\begin{Definition}\label{Definition 4.2}
We say that a path of initial data $\gamma^0:\theta\mapsto u_0^\theta$, $\theta\in[0,1]$ is a piecewise regular path if the following conditions hold

{\rm (i)} There exists a continuous map $(Y,\theta)\mapsto (x,u,\alpha,\xi)$ such that the semi-linear system \eqref{4.1}--\eqref{4.3} holds for $\theta\in[0,1]$, and the function $u^\theta(x,t)$ whose graph is 
\begin{equation*}
\text{ Graph }(u^\theta)=\{(t,x(t,Y,\theta),u(t,Y,\theta));~(t,Y)\in \R^+\times\mathbb{R}\}
\end{equation*}
provides the conservation solution of \eqref{2.1} with initial data $u^\theta(0,x)=u^\theta_0(x)$.

{\rm (ii)} There exist finitely many values $0=\theta_0<\theta_1<\cdots<\theta_N=1$ such that the map $(Y,\theta)\mapsto (x,u,\alpha,\xi)$ is $\mathcal{C}^\infty$ for $\theta\in(\theta_{i-1},\theta_i), i=1,\cdots, N$, and the solution $u^\theta=u^\theta(t,x)$ has only generic singularities at time $t=0$.

In addition, if for all $\theta\in[0,1]\backslash\{\theta_1,\cdots,\theta_N\}$, the solution $u^\theta$ has generic singularities for $t\in [0,T]$, then we say that the path of solution $\gamma^t: \theta\mapsto u^\theta$ is {\bf piecewise regular} for $t\in[0,T]$.
\end{Definition}


An application of Theorem \ref{thm_para} gives the following density result.
\begin{Corollary}\label{Corollary 4.1}
Given $T>0,$ let $\theta\mapsto(x^\theta, u^\theta,\alpha^\theta,\xi^\theta), \theta\in[0,1],$ be a smooth path of solutions to the system \eqref{4.1}--\eqref{4.3}. Then there exists a sequence of paths of solutions $\theta\mapsto (x^\theta_n, u^\theta_n,\alpha^\theta_n,\xi^\theta_n),$ such that 

{\rm (i)} For each $n\geq 1$, the path of the corresponding solution of \eqref{2.1} $\theta\mapsto u_n^\theta$ is regular for $t\in[0,T]$ in the sense of Definition \ref{Definition 4.2}.

{\rm (ii)} For any bounded domain $\Omega$ in the $(t$,$Y)$ space, the functions $(x^\theta_n,u^\theta_n,\alpha^\theta_n,\xi^\theta_n)$ converge to $(x^\theta, u^\theta,\alpha^\theta,\xi^\theta)$ uniformly in $\mathcal{C}^k([0,1]\times\Omega)$, for every $k\geq 1$, as $n\to \infty.$
\end{Corollary}

\section{Metric for general weak solutions}\label{sec_metric}
In this section, we will first extend the Lipschitz metric for smooth solutions in Section 3 to piecewise smooth
solutions, and then to general weak solutions using the generic regularity result established in Theorem \ref{thm_generic}. Through this Theorem \ref{thm_Lip geo} can be achieved. 
\subsection{Tangent vectors in transformed coordinates}\label{sec_transformed tangent}

For a reference solution $u(t,x)$ of \eqref{2.1} and a family of perturbed solutions $u^\varepsilon(t,x),$ we assume that, in the $(t$,$Y)$ coordinates, these define a family of smooth solutions of \eqref{4.1}--\eqref{4.3}, denoted by $(x^\varepsilon,u^\varepsilon,\alpha^\varepsilon,\xi^\varepsilon)$.

Consider the perturbed solutions of the form
\begin{equation}\label{3.7}
(x^\varepsilon,u^\varepsilon,\alpha^\varepsilon,\xi^\varepsilon)=(x,u,\alpha,\xi)+\varepsilon(X, U, A, \zeta)+o(\varepsilon).
\end{equation}
Since the coefficients of \eqref{4.1}--\eqref{4.3} are smooth, we have that the first order perturbations satisfy a linearized system and are well defined for $(t,Y)\in\mathbb{R}^+\times\mathbb{R}$. In the following, we will express the quantities in \eqref{2.17} in terms of $(X, U, A, \zeta)$. 

(1) The shift in $x$ is computed by
\begin{equation}\label{3.8}
w=\lim_{\varepsilon\to 0}\frac{x^\varepsilon(t,Y^\varepsilon)-x(t,Y)}{\varepsilon}
=\left. X+x_Y\cdot\frac{\partial Y^\varepsilon}{\partial \varepsilon}\RV_{\varepsilon=0}.
\end{equation}

(2) The change in $u$ is
\begin{equation}\label{3.9}
v+u_xw=\lim_{\varepsilon\to 0}\frac{u^\varepsilon(t,Y^\varepsilon)-u(t,Y)}{\varepsilon}
=\left. U+u_Y\cdot\frac{\partial Y^\varepsilon}{\partial \varepsilon}\RV_{\varepsilon=0}.
\end{equation}

(3) For the change in $\arctan u_x$, it holds that
\begin{equation}\label{3.10}
v_x+u_{xx}w=\frac{d}{d\varepsilon}\tan\frac{\alpha^\varepsilon(t,Y^\varepsilon)}{2}|_{\varepsilon=0}=\left. \frac{1}{2}\sec^2\frac{\alpha}{2} \LC A+\alpha_Y\cdot\frac{\partial Y^\varepsilon}{\partial \varepsilon}\RV_{\varepsilon=0}\RC.
\end{equation}

(4) Finally, we will derive an expression for the term $I_4$ in \eqref{2.17}. 
\begin{equation}\label{3.11}
\begin{split}
\left. \frac{d}{d\varepsilon}(\xi^\varepsilon(t,Y^\varepsilon)+\xi x_Y Y_x^\varepsilon)\RV_{\varepsilon=0}= \left. \zeta+\xi_Y\cdot\frac{\partial Y^\varepsilon}{\partial \varepsilon}\RV_{\varepsilon=0} + \left. \xi x_Y\cdot\frac{\partial Y_x^\varepsilon}{\partial \varepsilon}\RV_{\varepsilon=0},
\end{split}\end{equation}
where the change in the base measure with density $(1+u_x^2)^2$ is given by
\begin{equation*}
\begin{split}
\left. \frac{d}{d\varepsilon}\xi^\varepsilon\RV_{\varepsilon=0}=\lim_{\varepsilon\to 0}\frac{\xi^\varepsilon(t,Y^\varepsilon)-\xi(t,Y)}{\varepsilon}
=\left. \zeta+\xi_Y\cdot\frac{\partial Y^\varepsilon}{\partial \varepsilon}\RV_{\varepsilon=0}.
\end{split}\end{equation*}
Notice that \[(1+u^2_x)^2\,dx=\xi\,dY.\] In light of \eqref{3.8}--\eqref{3.11}, the weighted norm of a tangent vector \eqref{2.17} can be written as 
\begin{equation}\label{3.12}
\|(w, v)\|_u=\sum_{\ell=1}^4\int_{\mathbb{R}} |J_\ell(t,Y)|\,dY,
\end{equation}
where 
\begin{equation*}
\begin{split}
&J_1=\LC \left. X+x_Y\cdot\frac{\partial Y^\varepsilon}{\partial \varepsilon}\RV_{\varepsilon=0}\RC \xi  e^{-|x(t,Y)|},\quad \quad J_2=\LC \left. U+u_Y\cdot\frac{\partial Y^\varepsilon}{\partial \varepsilon}\RV_{\varepsilon=0}\RC \xi e^{-|x(t,Y)|},\\
&J_3=\frac{1}{2}\LC \left. A+\alpha_Y\cdot\frac{\partial Y^\varepsilon}{\partial \varepsilon}\RV_{\varepsilon=0}\RC \xi e^{-|x(t,Y)|},\ \ \ J_4=\LC\zeta+\left. \xi_Y\cdot\frac{\partial Y^\varepsilon}{\partial \varepsilon}\RV_{\varepsilon=0} + \left. \xi x_Y\cdot\frac{\partial Y_x^\varepsilon}{\partial \varepsilon}\RV_{\varepsilon=0}\RC e^{-|x(t,Y)|}.
\end{split}\end{equation*}
Since $Y^ \varepsilon(t)=Y^ \varepsilon(0)$ along the characteristic, it is easy to verify that the  integrands $J_\ell$ are smooth, for $\ell=1,2,3,4.$

\subsection{Length of piecewise regular paths}\label{sec_length}
Now we define the length of the piecewise regular path by optimizing over all transportation plans, namely,
\begin{Definition}\label{def5.1}
The length $\|\gamma^t\|$ of the piecewise regular path $\gamma^t: \theta\mapsto u^\theta(t)$ is defined as 
\begin{equation*}
    \|\gamma^t\|=\inf_{\gamma^t}\int_0^1\sum_{\ell=1}^4\int_{\mathbb{R}}|J_\ell^\theta(t,Y)|\,dY\,d\theta,
\end{equation*}
where the infimum is taken over all piecewise regular paths.
\end{Definition}
Our main result in this section is stated  as follows.
\begin{Theorem}\label{thm_length}
Given any $T>0$, consider a path of solutions $\theta\mapsto u^\theta$ of \eqref{2.1}, which is piecewise regular for $t\in[0,T]$. Moreover, the total energy $\|u^\theta\|^2_{H^1}$ 
and the norm $\|u^\theta_x\|_{L^4}^4$ is less than some constants $E_1>0$ and $E_2>0$, respectively . Then there exists some constant $C>0$, such that the length satisfies
\begin{equation}\label{6.1}
\|\gamma^t\|\leq C\|\gamma^0\|,
\end{equation}
where the constant $C$ depends only on $T, E_1$ and $E_2$.
\end{Theorem}

\begin{proof}
By the definition of piecewise regular paths,
we know that $u^\theta$ has generic regularities for every $\theta\in[0,1]\backslash\{\theta_1,\cdots,\theta_N\}$. So the solution $u^\theta$ is smooth in the $t$-$Y$ coordinates and piecewise smooth in the $t$-$x$ coordinates, thus the existence of the tangent vector is obvious. 

\paragraph{{\bf 1.}} We claim that, for $\theta\in[0,1]\backslash\{\theta_1,\cdots,\theta_N\}$,  if we have 
\begin{equation}\label{6.2}
\|v^\theta(t)\|_{u^\theta(t)}\leq e^{C_1t}\|v^\theta(0)\|_{u^\theta(0)},
\end{equation}
then \eqref{6.1} holds.
Here the constant $C_1$ depends only on the upper bound of the total energies.

Indeed, according to the definition of the length of a piecewise regular path, fix $\epsilon>0$, there exists some $Y$, such that, at time $t=0$,
\begin{equation*}
\int_0^1\sum_{\ell=1}^4\int_{\mathbb{R}} |J_\ell^\theta(0,Y)|\,dY\,d\theta\leq \|\gamma^0\|+\epsilon.
\end{equation*} 
Thus, integrating \eqref{6.2} over the interval $\theta\in[0,1]$, it holds that
\begin{equation*}
\|\gamma^t\|\leq C_2(\|\gamma^0\|+\epsilon),
\end{equation*}
which implies \eqref{6.1}, since $\epsilon>0$ is arbitrary.

\paragraph{{\bf 2.}}
Therefore in the following, we will concentrate on proving \eqref{6.2}.  If $u^\theta$ is smooth in the $(t, x)$ variables, in light of \eqref{2.19}, \eqref{6.2} follows directly. Thus, it suffices to show that the same result can be applied if $u^\theta$ is piecewise smooth with generic singularities. 

We first note that there exist at most finitely many points $Q_j=(t_j,Y_j), j=1,\cdots,N$ on which $\alpha^\theta=\pi$, $\alpha^\theta_Y=0$, when $t\in[0,T]$ . In fact,  using the implicit function theorem and the generic condition (\ref{generic_con}), we know that each point $Q$ with $\alpha^\theta=\pi$ and $\alpha^\theta_Y=0$ is isolated from other $Q's$ in a small neighborhood. Then it follows that there exist at most finite many points $Q_j$ by applying a finite covering argument on the bounded region including all possible singularities (the existence of such a bounded region is shown in Part 1 of the proof of Theorem \ref{thm_generic} in Section \ref{sub4.2}.).

Now, for each time $t=t_j$ corresponding to the point $Q_j$, the map
\begin{equation*}
t\mapsto\int_0^1\sum_{\ell=1}^4\int_{\mathbb{R}} |J_\ell^\theta(t,Y)|\,dY\,d\theta
\end{equation*}
is continuous,
Since the number of $t_j$ is finite,
the metric will not be affected. 

At time $t\neq t_j$, on the other hand, we claim that the generic singularity does not affect the estimate \eqref{2.19}, that is,  the effect of the generic singularity to the time derivative
\begin{equation*}
\frac{d}{dt}\sum_{\ell=1}^4\int_{\mathbb{R}} |J_\ell^\theta(t,Y)|\,dY\end{equation*} 
is negligible. 

We know that
\beq\label{alpialyn0}
\alpha^\theta=\pi,\qquad\alpha^\theta_Y\neq0
\eeq
at the singularities.
For a fixed time $\tau$, let the point $(t_\varepsilon,Y_\varepsilon)$ be the intersection of $\Gamma_{\tau-\varepsilon}=\{(t,Y);~ t=\tau-\varepsilon\}$ and $\{(t,Y); ~\alpha^\theta (t,Y)=\pi\}$, and the point $(t'_\varepsilon,Y'_\varepsilon)$ be the intersection of $\Gamma_{\tau+\varepsilon}=\{(t,Y);~ t=\tau+\varepsilon\}$ and $\{(t,Y); ~\alpha^\theta (t,Y)=\pi\}$. Then denote 
\begin{equation*}
\begin{cases}
&\Lambda_\varepsilon^+:=\Gamma_{\tau+\varepsilon}\cap\{(t,Y);Y\in [Y'_\varepsilon,Y_\varepsilon]\},\\
&\Lambda_\varepsilon^-:=\Gamma_{\tau-\varepsilon}\cap\{(t,Y);Y\in [Y'_\varepsilon,Y_\varepsilon]\}.
\end{cases}
\end{equation*}
Thus, we have
\begin{equation*}
\lim_{\varepsilon\to 0}\frac{1}{\varepsilon}\Big(\int_{\Lambda_\varepsilon^+}-\int_{\Lambda_\varepsilon^-}\Big)\sum_{\ell=1}^4 |J_\ell^\theta(t,Y)|\,dY=0,
\end{equation*}
since each integrand is continuous and $|Y_\varepsilon -Y'_\varepsilon|=O(\varepsilon)$ because of (\ref{alpialyn0}). This means that the curve  $\{(t,Y); \alpha^\theta (t,Y)=\pi\}$ has non-horizontal tangent line at the singularity. This implies that the estimate \eqref{2.19} remains valid in the presence of singular curve where $\alpha^\theta=\pi$. Hence we complete the proof of Theorem \ref{thm_length}.
\end{proof}

\begin{Remark}
We want to point out that the assumption that the solution path in Corollary \ref{Corollary 4.1} is regular, i.e. 
the solution has only generic singularities, except at finitely many $\theta$ values, is very crucial. Because of this property we show that there are only finitely many points $Q_j$ with $\alpha=\pi$, $\alpha_Y=0$. However at time $t = t_j$ we cannot treat these $Q_j$'s directly by the method used for $t \neq t_j$ since now $|Y_\varepsilon -Y'_\varepsilon|=O(1)$.  
\end{Remark}

\subsection{Construction of the geodesic distance}\label{sec_geodesic}

In light of Theorem \ref{thm_generic}, there exists an open dense set $\mathcal{D}\subset \Big(\mathcal{C}^3(\mathbb{R})\cap H^1(\mathbb{R})\cap W^{1,4}(\mathbb{R})\Big)$, such that, for $u_0\in\mathcal{D}$, the solution of \eqref{2.1} has only generic singularities. Now, on $\mathcal{D}^\infty:=\mathcal{C}_0^\infty\cap\mathcal{D}$, we construct a geodesic distance, defined as the infimum among the weighted lengths of all piecewise regular paths connecting two given points.

Consider two solutions $u,\tilde{u}\in \mathcal{D}^\infty$. Denote their total energies as
\begin{equation*}
\mathcal{E}(u):=\int_\R (u^2+u_{x}^2)(x)\,dx,\quad \mathcal{E}(\tilde{u}):=\int_\R (\tilde{u}^2+\tilde{u}_{x}^2)(x)\,dx,
\end{equation*}
respectively. For $E_1>0$ and $E_2>0$ denote the set
\begin{equation*}
\Sigma:=\{u\in H^1(\mathbb{R})\cap W^{1,4}(\mathbb{R}); ~\mathcal{E}(u)\leq E_1, \| u_{x}\|_{L^4}^4\leq E_2\}.
\end{equation*}

\begin{Definition}\label{Definition 7.1}
For solutions with initial data in $\mathcal{D}^\infty\cap \Sigma$, we define the geodesic distance $d(u,\tilde{u})$ as the infimum among the weighted lengths of all piecewise regular paths, which connect $u$ with $\tilde{u}$, that is, for any time $t$,
\begin{equation*}
\begin{split}
d(u,\tilde{u}):=\inf \LCB \|\gamma^t\|:\  \right. &\gamma^t \text{ is a piecewise regular path}, \gamma^t(0)=u,\gamma^t(1)=\tilde{u},\\
&\left. \mathcal{E}(u^\theta)\leq E_1, \| u_{x}^\theta\|_{L^4}^4\leq E_2, \text{ for all } \theta\in[0,1]\RCB.
\end{split}\end{equation*} 
\end{Definition}

Finally we can define the metric for the general weak solutions.

\begin{Definition}\label{Definition 7.2}
Let $u_0$ and $\tilde{u}_0$ in $H^1(\R)\cap W^{1,4}(\R)$ be two absolute continuous initial data as required in the existence and uniqueness Theorem \ref{thm_exist}. Denote $u$ and $\tilde{u}$ to be the corresponding global weak solutions, then we define, for any time $t$,
\begin{equation*}
d(u,\tilde{u}):=\lim_{n\rightarrow \infty}d(u^n,\tilde{u}^n),
\end{equation*} 
for any two sequences of solutions $u^n$ and ${\tilde u}^n$ in $\mathcal{D}^\infty\cap \Sigma$ with
\[
\|u^n-u\|_{H^1\cap W^{1,4}}\rightarrow 0,\quad\hbox{and}\quad
\|\tilde{u}^n-\tilde{u}\|_{H^1\cap W^{1,4}}\rightarrow 0.
\]
\end{Definition}

The limit in the definition is independent of the choice of sequences, because the solution flows are Lipschitz  in $\mathcal{D}^\infty\cap \Sigma$. Since the concatenation of two piecewise regular paths is still a piecewise regular path (after a suitable re-parameterization), it is clear that $d(\cdot,\cdot)$ is a distance. This way the metric is well-defined.

Note that when 
\[\|u^n_0-u_0\|_{H^1\cap W^{1,4}}\rightarrow 0,\] 
it is easy to show that the corresponding solutions satisfy, for any $t>0$,
\[\|u^n-u\|_{H^1\cap W^{1,4}}\rightarrow 0\]
by the semi-linear equations (\ref{4.1}). So it is clear that the Lipschitz property in Theorem \ref{thm_length} can be extended to the general solutions, and so we conclude to obtain Theorem \ref{thm_Lip geo}. 

\section{Comparison with other metrics}\label{sec_comparison}

The purpose of this section is to compare the distance $d(\cdot,\cdot)$ with other types of metrics. 
\begin{Proposition}[Comparison with the Sobolev metric]\label{Proposition 7.1}
For any solutions $u, \tilde{u}\in  \Sigma $ 
to (\ref{2.1}), there exists some constant $C$ depends only on $E_1$ and $E_2$, such that, 
\begin{equation*}
d(u,\tilde{u})\leq C\Big(\|u-\tilde{u}\|_{H^1}+\|(u-\tilde{u})e^{-|x|}\|_{L^{1}}+\|(u_x-\tilde{u}_x)e^{-|x|}\|_{L^1}+\|u_{x}-\tilde{u}_{x}\|_{L^4}\Big).
\end{equation*}
\end{Proposition}

\begin{proof}
Without loss of generality, we only consider the solutions in
$\mathcal{D}^\infty\cap\Sigma$.

For $\theta\in[0,1]$, consider the interpolated data $u^\theta$ as
\begin{equation}\label{7.1}
u^\theta=\theta\tilde{u}+(1-\theta)u.
\end{equation}
Then the total energy of $u^\theta$ satisfies
\begin{equation}\label{7.2}
\begin{split}
&\int_\R \Big((u^\theta)^2+(u^\theta_{x})^2\Big)(t,x)\,dx=\int_\R \Big([\theta\tilde{u}+(1-\theta)u]^2+[\theta\tilde{u}_{x}+(1-\theta)u_{x}]^2\Big)\,dx\\
\leq & \int_\R \Big([\theta^2+\theta(1-\theta)](\tilde{u}^2+\tilde{u}_{x}^2)+[(1-\theta)^2+\theta(1-\theta)](u^2+u_{x}^2)\Big)\,dx\\
\leq &\max\{\mathcal{E}(u),\mathcal{E}(\tilde{u})\}\leq E_1.
\end{split}
\end{equation}
Also, for the $L^4$ norm of $u^\theta_{x}$, it holds that
\begin{equation}\label{7.3}
\begin{split}
\int_\R (u^\theta_{x})^4(t,x)\,dx 
\leq & \int_\R \big(\theta\tilde{u}_{x}^4+(1-\theta)u_{x}^4\big)\,dx
\leq \max\left\{ \int_\R u_{x}^4\ dx, \int_\R \tilde{u}_{x}^4\ dx \right\}\leq E_2.
\end{split}
\end{equation}

Now, we will estimate the weighted length of the path $\gamma^t:\theta\mapsto u^\theta$ in \eqref{7.1}. The goal is to show that 
\begin{equation*}
\|\gamma^t\|\leq C\Big(\|u-\tilde{u}\|_{H^1}+\|(u-\tilde{u})e^{-|x|}\|_{L^{1}}+\|(u_x-\tilde{u}_x)e^{-|x|}\|_{L^1}+\|u_{x}-\tilde{u}_{x}\|_{L^4}\Big)
\end{equation*}
for some constant $C$ depends only on $E_1$ and $E_2$. First, from \eqref{7.1}, we obtain
\begin{equation}\label{7.4}
v^\theta=\frac{d u^\theta}{d\theta}=\tilde{u}-u, \quad u_{x}^\theta=\theta\tilde{u}_{x}+(1-\theta)u_{x}.
\end{equation}
To derive an upper bound for the weighted length $\|\gamma^t\|$, we can choose the shift $w=0$ in \eqref{2.17}. Indeed, by \eqref{7.2}--\eqref{7.4} and the definition of the weighted length of the path $\gamma^t$, we have
\begin{equation*}
\begin{split}
\|\gamma^t\|=&\int_0^1 \|v^\theta\|_{u^\theta}\,d\theta\\
=&\int_0^1\int_\R \Big(|v^\theta|(1+(u^\theta_{x})^2)^2+|v^\theta_{x}|(1+(u^\theta_{x})^2)+4|\big(u^\theta_{x}+(u^\theta_{x})^3\big) v^\theta_{x}|\Big)e^{-|x|}\,dx\,d\theta\\
\leq& \|(\tilde{u}-u)e^{-|x|}\|_{L^1}+ C\|\tilde{u}-u\|_{L^\infty}\int_0^1(\|u^\theta_{x}\|_{L^2}^2+\|u^\theta_{x}\|_{L^4}^4)\,d\theta+\|(\tilde{u}_{x}-u_{x})e^{-|x|}\|_{L^1}\\
&+ \|\tilde{u}_{x}-u_{x}\|_{L^2}\int_0^1\|u^\theta_{x}\|_{L^4}^2\,d\theta+ C\|\tilde{u}_{x}-u_{x}\|_{L^2}\int_0^1\|u^\theta_{x}\|_{L^2}\,d\theta\\
&+ C\|\tilde{u}_{x}-u_{x}\|_{L^4}\int_0^1\|u^\theta_{x}\|_{L^4}^3\,d\theta\\
\leq &C\Big(\|u-\tilde{u}\|_{H^1}+\|(u-\tilde{u})e^{-|x|}\|_{L^{1}}+\|(u_x-\tilde{u}_x)e^{-|x|}\|_{L^1}+\|u_{x}-\tilde{u}_{x}\|_{L^4}\Big).
\end{split}
\end{equation*}
Here, $C$ denotes a generic positive constant depends only on $E_1$ and $E_2$. This completes the proof of Proposition \ref{Proposition 7.1}.
\end{proof}

In the following two propositions, we compare the distance with the $L^1$ distance and the Kantorovich-Rubinstein or Wasserstein distance.

\begin{Proposition}[Comparison with $L^1$ metric]\label{Proposition 7.2}
For any solutions $u, \tilde{u}\in H^1(\R)\cap W^{1,4}(\R)$ to (\ref{2.1}), there exists some constant $C$ depends only on $E_1$ and $E_2$, such that, 
\begin{equation}\label{7.5}
\|(u-\tilde{u})e^{-|x|}\|_{L^1}\leq C\cdot d(u,\tilde{u}).
\end{equation}
\end{Proposition}

\begin{proof}   Assume that $\gamma^t:\theta\mapsto u^\theta$ is a regular path connecting $u$ with $\tilde{u}$.

{\bf 1.} Notice that
\begin{equation*}
|v|=|v+u_x w-u_x w|\leq |v+u_x w|+|u_x w|\leq |v+u_x w|(1+u_x^2)^2+|w|(1+u_x^2)^2.
\end{equation*}
Thus, by the above inequality, the definition \ref{Definition 7.1}, \eqref{2.17} and \eqref{7.4}, for some constants $C_3, C_4>0$, we have
\begin{equation*}
d(u,\tilde{u})\geq C_3\inf_{\gamma^t}\int_0^1\int_\R |v^\theta|e^{-|x|}\,dx\,d\theta=C_3\inf_{\gamma^t}\int_0^1\int_\R |\frac{d u^\theta}{d\theta}|e^{-|x|}\,dx\,d\theta\geq C_4\|(u-\tilde{u})e^{-|x|}\|_{L^1},
\end{equation*}
which implies \eqref{7.5}.
\end{proof}

\begin{Proposition}[Comparison with the Kantorovich-Rubinstein metric]\label{Proposition 7.3}
Consider the same assumptions as in Proposition  \ref{Proposition 7.2}, and further assume that $u, \tilde{u}\in L^1(\R)$.
Then one could drop the $e^{-|x|}$ term in \eqref{2.17} to define a new distance $d^*$, such that 
\begin{equation}\label{7.52}
\|(u-\tilde{u})\|_{L^1}\leq C\cdot d^*(u,\tilde{u}), \quad \text{ and }
\end{equation}
\begin{equation}\label{7.6}
\sup_{\|f\|_{\mathcal{C}^1}\leq 1}\left|\int f \,d\mu-\int f\,d\tilde{\mu}\right|\leq d^*(u,\tilde{u}),
\end{equation}
where $\mu,\tilde{\mu}$ are the measures with densities $(1+(u_{x})^2)^2$ and $(1+(\tilde{u}_{x})^2)^2$ w.r.t the Lebesgue measure, respectively.
\end{Proposition}
\begin{proof}

First, repeating the same proof as in Proposition \ref{Proposition 7.2} without the term $e^{-|x|}$, one can very easily obtain \eqref{7.52}.

Then, for any function $f$, such that $\|f\|_{\mathcal{C}^1}\leq 1$, denote  $\mu^\theta$ to be the measure with density $(1+(u_{x}^\theta)^2)^2$ w.r.t the Lebesgue measure, then the following holds
\begin{equation}\label{7.7}
\begin{split}
\left|\int_0^1\frac{d}{d\theta}\int f\,d\mu^\theta\,d\theta \right| \leq& \int_0^1\int_\R |f'|\cdot |w^\theta|(1+(u_{x}^\theta)^2)^2\,dx\,d\theta\\
&+\int_0^1\int_\R |f|\cdot|4(u_{x}^\theta+(u_{x}^\theta)^3)(v_{x}^\theta+u_{xx}^\theta w^\theta)+(1+u_{x}^\theta)^2)^2 w^\theta_x|\,dx\,d\theta,
\end{split}\end{equation}
where the two integrals on the right hand side of \eqref{7.7} are exactly $I_1$ and $I_4$ of \eqref{2.17}. Hence, we get \eqref{7.6} immediately. This completes the proof of Proposition \ref{Proposition 7.2}.
\end{proof}

The metric (\ref{7.6}) is usually called a Kantorovich-Rubinstein distance, which is equivalent to a Wasserstein distance by a duality theorem \cite{V}.

\section{Application to the Camassa-Holm equation} \label{appendix CH}

Our method of constructing the geodesic distance can be systematically applied to many other quasilinear equations. Here we give an example of the Camassa-Holm (CH) equation. The energy measure associated to the CH equation has density $u_x^2$. Here in the Section we will only outline the construction of the Finsler norm for the infinitesimal tangent vector for smooth solutions. The issues on the generic regularity (already given in \cite{LZ}) and the construction of the geodesic distance can be treated in a similar way as is done for Novikov equation in this paper.

The CH equation reads
  \begin{equation}\label{A.1}
\begin{cases}
      \displaystyle u_t+\LC\frac{u^2}{2}\RC_x+P_x=0,\\
       u(0,x)=u_0(x),
        \end{cases}
 \end{equation}
where 
\begin{equation*}
P:=\frac{1}{2}e^{-|x|}*\LC u^2+\frac{u_x^2}{2}\RC.
\end{equation*}

Let $u(x)$ be a smooth solution to \eqref{A.1} and consider a family of perturbed solutions
\begin{equation*}
u^\epsilon(x)=u(x)+\epsilon v(x)+o(\epsilon).
\end{equation*}

A direct calculation yields that the first order perturbation $v$ satisfies
\begin{equation}\label{A.2}
v_t+uv_x+vu_x+\frac{1}{2}\LC \int_x^{\infty}-\int^x_{-\infty}\RC e^{-|x-y|}(2uv+u_yv_y)(y)\,dy=0.
\end{equation}
\begin{equation}\label{A.3}
v_{xt}+uv_{xx}+u_xv_x+vu_{xx}-2uv+\frac{1}{2}\int_\R e^{-|x-y|}(2uv+u_yv_y)(y)\,dy=0.
\end{equation}

Similar as before, we introduce a horizontal shift $w$ satisfying
\[
x^\epsilon=x+\epsilon w+o(\epsilon)
\]
Again the shift component of the tangent vector must propagate along characteristic. Namely
\begin{equation}\label{A.5}
w_t+uw_x=v+u_x w.
\end{equation}

Now we can define a Finsler norm
\begin{equation*}
\|v\|_u=\inf_{w\in\mathcal{A}}\|(w, \hat v)\|_u, \quad \text{ with }\quad \hat v = v + u_x w, 
\]
where 
\[
\qquad\mathcal{A} = \LCB {{\hbox{solutions $w(t,x)$ of \eqref{A.5} with  smooth initial data $w_0(x)$} }}  \RCB.
\]
and
\begin{equation*}
\begin{split}
\|(w,\hat v)\|_u=&\int_\R\{[\text{change in } x]+[\text{change in } u]\}(1+u_x^2)\,dx\\
&+\int_\R [\text{change in the base measure with density } u_x^2]\,dx.
\end{split}
\end{equation*}
More precisely, we have
\begin{equation}\label{A.4}
\begin{split}
\quad\|v\|_u 
&=\inf_{w}\int_\R \{|w|(1+u_x^2)+|v+u_x w|(1+u_x^2)+|2u_x(v_x+u_{xx}w)+u_x^2w_x|\}e^{-|x|}\,dx\\
&=\inf_{w} \LC I_1+I_2+I_3 \RC.
\end{split}\end{equation}

The goal of the forthcoming computations is to validate the estimate
\begin{equation}\label{A.6}
\frac{d}{dt}\|v(t)\|_{u(t)}\leq C\|v(t)\|_{u(t)},
\end{equation}
for some constant $C$ depending only on the total energy. As same as before, in the following calculation,
we drop the $e^{-|x|}$ terms.

{\bf 1.} To estimate the time derivative of $I_1$, by \eqref{A.1} and \eqref{A.5}, we have
\begin{equation*}
\begin{split}
\big(w(1+u_x^2)\big)_t+\big(uw(1+u_x^2)\big)_x & =(w_t+uw_x)(1+u_x^2)+w[(1+u_x^2)_t+(u(1+u_x^2))_x]\\
& =(v+u_xw)(1+u_x^2)+w(u_x+2u^2u_x-2u_xP).
\end{split}
\end{equation*}
This yields the estimate
\begin{equation}\label{A.7}
\begin{split}
{dI_1\over dt} = \frac{d}{dt}\int_\R|w|(1+u_x^2)\,dx & \leq \int_\R |v+u_xw|(1+u_x^2)\,dx+C\int_\R |w|(1+u_x^2)\,dx \\
& \leq C(I_2 + I_1).
\end{split}
\end{equation}

{\bf 2.} To estimate the time derivative of $I_2$, recalling \eqref{A.1}, \eqref{A.2} and \eqref{A.5}, we obtain
\begin{equation*}
\begin{split}
&\quad\LC (v+u_xw)(1+u_x^2)\RC_t+\LC u(v+u_xw)(1+u_x^2)\RC_x\\
& = \LB v_t+uv_x+u_x(w_t+uw_x)+w(u_{xt}+uu_{xx})\RB (1+u_x^2)+(v+u_xw)\LB (1+u_x^2)_t+(u(1+u_x^2))_x \RB\\
& = \LB -vu_x-\frac{1}{2}\LC \int_x^{\infty}-\int^x_{-\infty}\RC e^{-|x-y|}(2uv+u_yv_y)\,dy+u_x(v+u_xw) \right. \\
& \quad \quad \left. +w\LC -\frac{u_x^2}{2}+u^2-P \RC \RB (1+u_x^2) +(v+u_xw)(u_x+2u^2u_x-2u_xP)\\
& = \LB -\frac{1}{2}\LC \int_x^{\infty}-\int^x_{-\infty}\RC e^{-|x-y|}(2uv+u_yv_y)\,dy+\frac{1}{2}u_x^2w+w(u^2-P)\RB (1+u_x^2)\\
& \quad \ +(v+u_xw)(u_x+2u^2u_x-2u_xP).
\end{split}
\end{equation*}
Note that $$2uv=2u(v+u_y w)-2u u_y w,$$ and
\begin{equation*}
\begin{split}
\frac{1}{2}u_x^2w=&-\frac{1}{2}\LC \int_x^{\infty}-\int^x_{-\infty}\RC \LC e^{-|x-y|}\frac12{u_y^2}w\RC_y\,dy\\
=& -\frac{1}{2}\LC \int_x^{\infty}-\int^x_{-\infty}\RC e^{-|x-y|}(u_yu_{yy}w+\frac12{u^2_y}w_y)\,dy+\frac{1}{2}\int_\R e^{-|x-y|}\frac12{u_y^2}w\,dy.
\end{split}
\end{equation*}
We thus conclude
\begin{equation}\label{A.8}
\begin{split}
{dI_2 \over dt} & = \frac{d}{dt}\int_\R |v+u_xw|(1+u_x^2)\,dx \\
& \leq C\LC \int_\R |2u_x(v_x+u_{xx}w)+u_x^2 w_x|\,dx+\int_\R |v+u_xw|(1+u_x^2)\,dx+\int_\R|w|(1+u_x^2)\,dx\RC \\
& \leq C(I_3 + I_2 + I_1).
\end{split}
\end{equation}

{\bf 3.} To estimate the time derivative of $I_3$, using \eqref{A.1}, \eqref{A.3} and \eqref{A.5} to get
\begin{equation}\label{A.9}
\begin{split}
&\quad \LC 2u_x(v_x+u_{xx}w)+u_x^2 w_x\RC_t+\LC u(2u_x(v_x+u_{xx}w)+u_x^2 w_x)\RC_x\\
& = 2(u_{xt}+uu_{xx})(v_x+u_{xx}w)+2u_x\LB v_{xt}+(uv_x)_x\RB +2u_xu_{xx}(w_t+uw_x)\\
&\quad +2u_xw\LB u_{xxt}+(uu_{xx})_x\RB +2u_xw_x(u_{xt}+uu_{xx})+u_x^2\LB w_{xt}+(uw_x)_x\RB \\
& = 2\LC -\frac{u_x^2}{2}+u^2-P\RC(v_x+u_{xx}w+u_xw_x)-2u_x\LC vu_{xx}-2uv+\frac{1}{2}\int_\R e^{-|x-y|}(2uv+u_yv_y)\,dy\RC \\
&\quad +2u_xu_{xx}(v+u_x w)+2u_xw(-u_xu_{xx}+2uu_x-P_x)+u_x^2(v_x+u_xw_x+u_{xx}w)\\
& = 2(u^2-P)(v_x+u_{xx}w+u_xw_x)+4uu_x(v+u_x w) -u_x\int_\R e^{-|x-y|}(2uv+u_yv_y)\,dy-2u_x wP_x .
\end{split}
\end{equation}
The first term in the last equality can be estimated as
\begin{equation}\label{A.10}
\begin{split}
& 2\int_\R (u^2-P)(v_x+u_{xx}w+u_xw_x)\,dx \\
& = 2\int_\R (u^2-P)(v+u_{x}w)_x\,dx =-2\int_\R (u^2-P)_x(v+u_{x}w)\,dx\\
&\leq C\int_\R |v+u_{x}w|(1+u_x^2).
\end{split}\end{equation}
For the third term, observe that 
\begin{equation}\label{A.11}2uv+u_yv_y=2u(v+u_y w)-2uu_y w+\frac{1}{2}\LB 2u_y(v_y+u_{yy}w)+u_y^2 w_y \RB-\frac{1}{2}(u^2_yw)_y.\end{equation}
We have
\begin{equation}\label{A.12}
\begin{split}
\LV \frac{1}{2}\int_\R \int_\R e^{-|x-y|}(u^2_yw)_y\,dy\,dx \RV & \leq C \int_\R \LC \int_\R|(e^{-|x-y|})_y|\,dx \RC |w|u^2_y\,dy\\
& \leq C\int_\R |w|(1+u_x^2)\,dx.
\end{split}
\end{equation}
In light of \eqref{A.9}--\eqref{A.12}, we obtain
\begin{equation}\label{A.13}
\begin{split}
{dI_3\over dt} & = \frac{d}{dt}\int_\R |2u_x(v_x+u_{xx}w)+u_x^2 w_x|\,dx \\
& \leq C \LC \int_\R |2u_x(v_x+u_{xx}w)+u_x^2 w_x|\,dx + \int_\R |v+u_xw|(1+u_x^2)\,dx + \int_\R |w|(1+u_x^2)\,dx \RC \\
& \leq C(I_3 + I_2 + I_1).
\end{split}
\end{equation}

Combining the inequalities \eqref{A.7}, \eqref{A.8} and \eqref{A.13} together, we obtain the desired inequality \eqref{A.6}.

\section{Interaction of two peakons.}\label{sec_peakon}
In this part, we use numeric method to study the interaction of two peakons, in the form of \eqref{intro_peakon},
for the Novikov equation. Especially, in this example, we show the energy concentration, which indicates the failure of $W^{1,4}(\mathbb{R})$ space in studying the Lipschitz continuous dependence, or in another word the necessity in using the transport metric.

As an example, we consider the following two-peakon initial data
\begin{equation}
\label{apbid}u(0, x) = \sum^2_{i=1} p_i(0) e^{|x - q_i(0)|}
\end{equation}
where $p_1=1$, $p_2=-0.5$, $q_1=-0.5$ and $q_2=0.5$, as shown in Figure \ref{appb_1}.

\begin{figure}[hb]
  \includegraphics[scale=0.5]{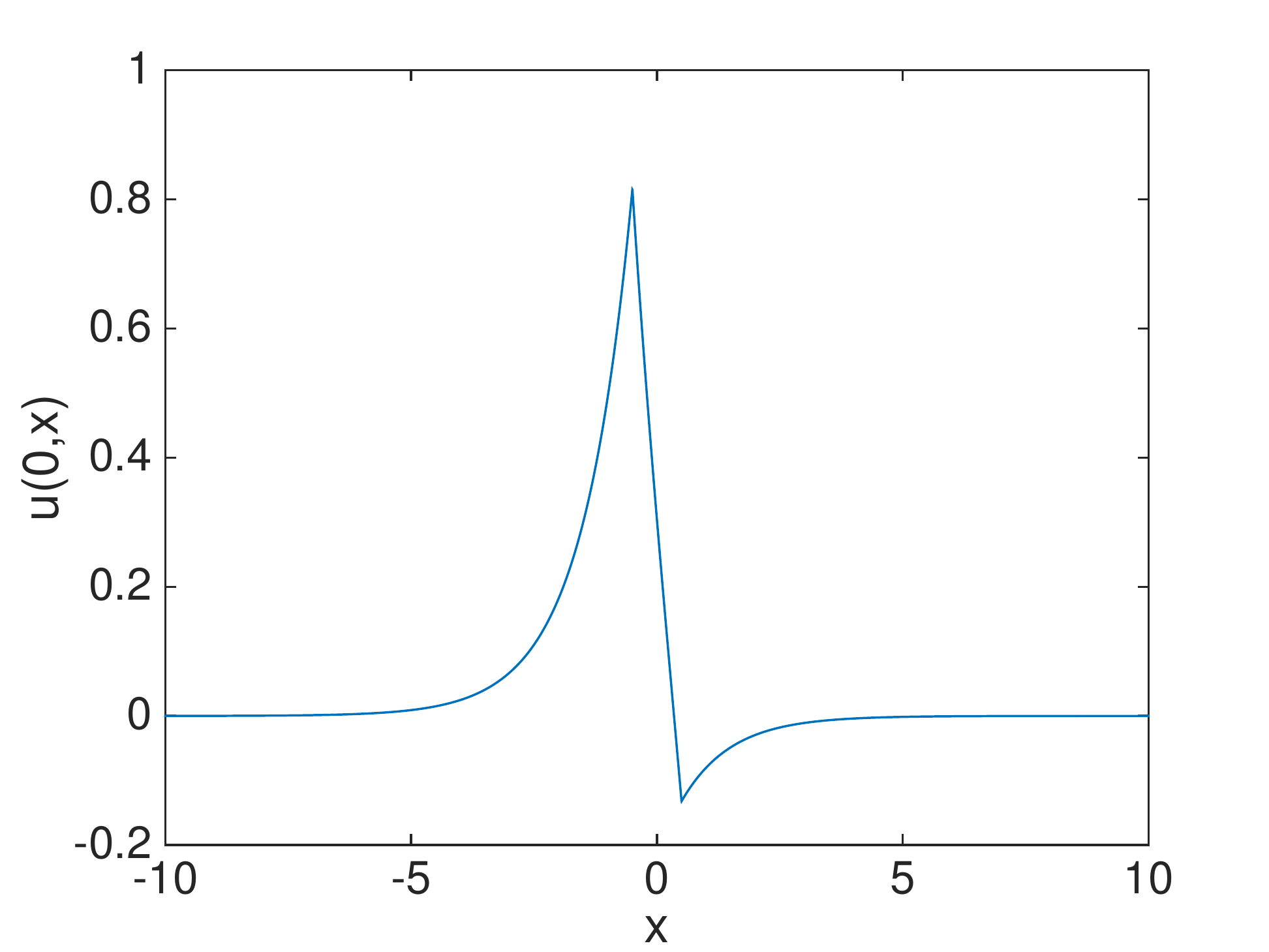}
  \caption{Two peakons: Initial data.}
  \label{appb_1}
\end{figure}

By studying the dynamical system (\ref{HT}) of $p_i$ and $q_i$ in a numeric way, we can clearly see, from Figure \ref{appb_2}, that
two peakons will interact. Similar simulation can be found in \cite{MK}. 

\begin{figure}[hb]
  \includegraphics[scale=0.5]{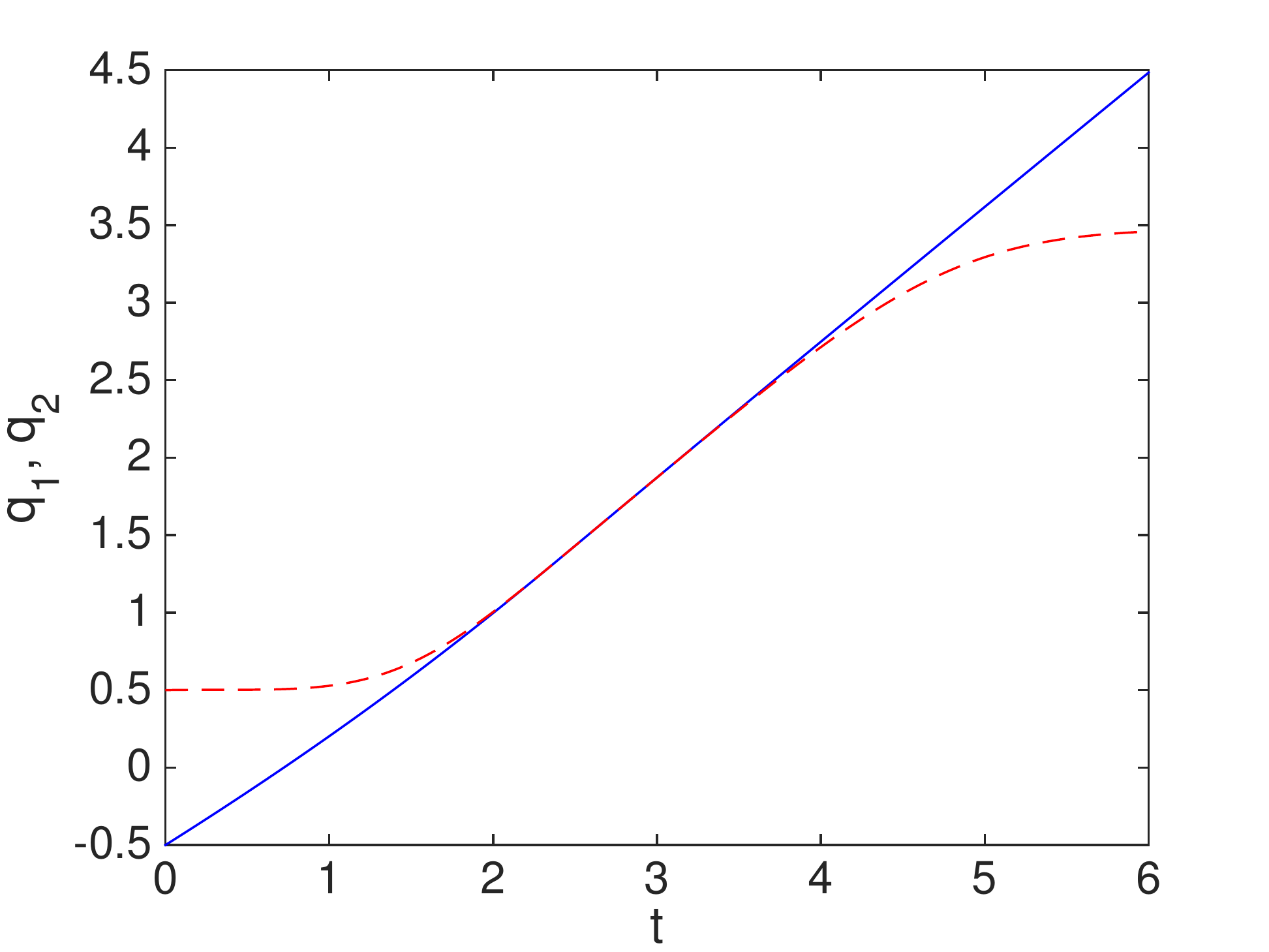}
  \caption{Interaction of two peakons on $(t,x)$-plane. The blue line is for the characteristic $q_1(t)$ and the red dash line is for the characteristic $q_2(t)$.}
  \label{appb_2}
\end{figure}


In order to 
know whether there is any energy concentration when two peakons interact,
one needs to study the energies
$\mathcal{E}$ and $\mathcal{F}$ in the interval $[q_1(t), q_2(t)]$ near and at the time when two peakons collide. 
Here $x=q_1(t)$ and $x=q_2(t)$ are two characteristics starting from the tips of two initial peakons, respectively.
However, the analysis becomes very hard at the collision if one only uses the system (\ref{HT}), because of the blowup of $p_i$.

Instead, in this paper, we use the semi-linear system established in \cite{CCL}, which dilates the interacting characteristics 
in the new $(t,Y)$-coordinates. 
In fact, integrating (\ref{4.1}), one has 
\begin{equation}\label{int4.1}
\begin{cases}
& u(t,Y)=u(0,Y)- \int_0^t \partial_x P_1+ P_2\  dt,\\
& \alpha(t,Y)=\alpha(0,Y) +\int_0^t 2u^3\cos^2\frac{\alpha}{2}-u\sin^2\frac{\alpha}{2}-2\cos^2\frac{\alpha}{2}(P_1+\partial_x P_2)\ dt,\\
& \xi(t,Y)=\xi(o,Y)+\int_0^t \xi[(2u^3+u)-2(P_1+\partial_x P_2)]\sin \alpha\ dt.
\end{cases}
\end{equation}
We use an iteration:
\begin{equation}
\begin{cases}
& u_{n+1}(t,Y)=u(0,Y)- \int_0^t \partial_x P_1(u_n, \alpha_n, \xi_n)+ P_2(u_n, \alpha_n, \xi_n)\  dt,\\
& \alpha_{n+1}(t,Y)=\alpha(0,Y) +\int_0^t 2u_n^3\cos^2\frac{\alpha_n}{2}-u_n\sin^2\frac{\alpha_n}{2}-2\cos^2\frac{\alpha_n}{2}(P_1+\partial_x P_2)(u_n, \alpha_n, \xi_n)\ dt,\\
& \xi_{n+1}(t,Y)=\xi(0,Y)+\int_0^t \xi[(2u_n^3+u_n)-2(P_1+\partial_x P_2)(u_n, \alpha_n, \xi_n)]\sin \alpha_n\ dt.
\end{cases}
\end{equation}
with initial data
\[
u(0,Y),\quad \alpha(0,Y)\quad \hbox{and}\quad\xi(0,Y)
\]
calculated by \eqref{apbid} and \eqref{aluxi}.
Here $n$ is the iteration index. The graphs of $u$ and $\alpha$ are given in Figures \ref{appb_3_0} and \ref{appb_3}, respectively.

The convergence of this algorithm can be proved by a similar method as the one used in \cite{CCL}. This is
one of the greatest advantages of our algorithm.

\begin{figure}[hb]
 \includegraphics[scale=0.5]{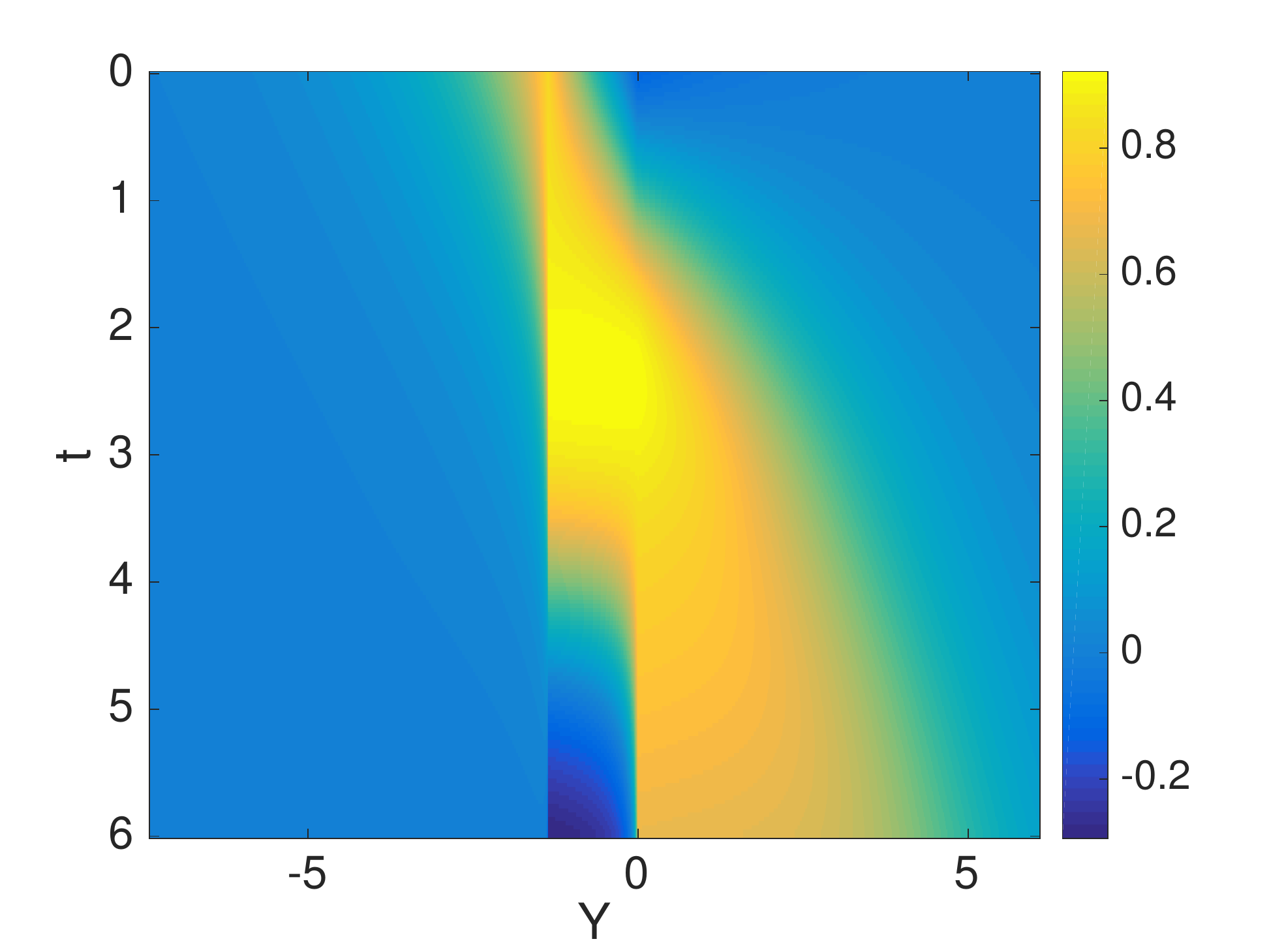}
  \caption{Solution $u(t,Y)$.}
  \label{appb_3_0}
\end{figure}

\begin{figure}[hb]
  \includegraphics[scale=0.5]{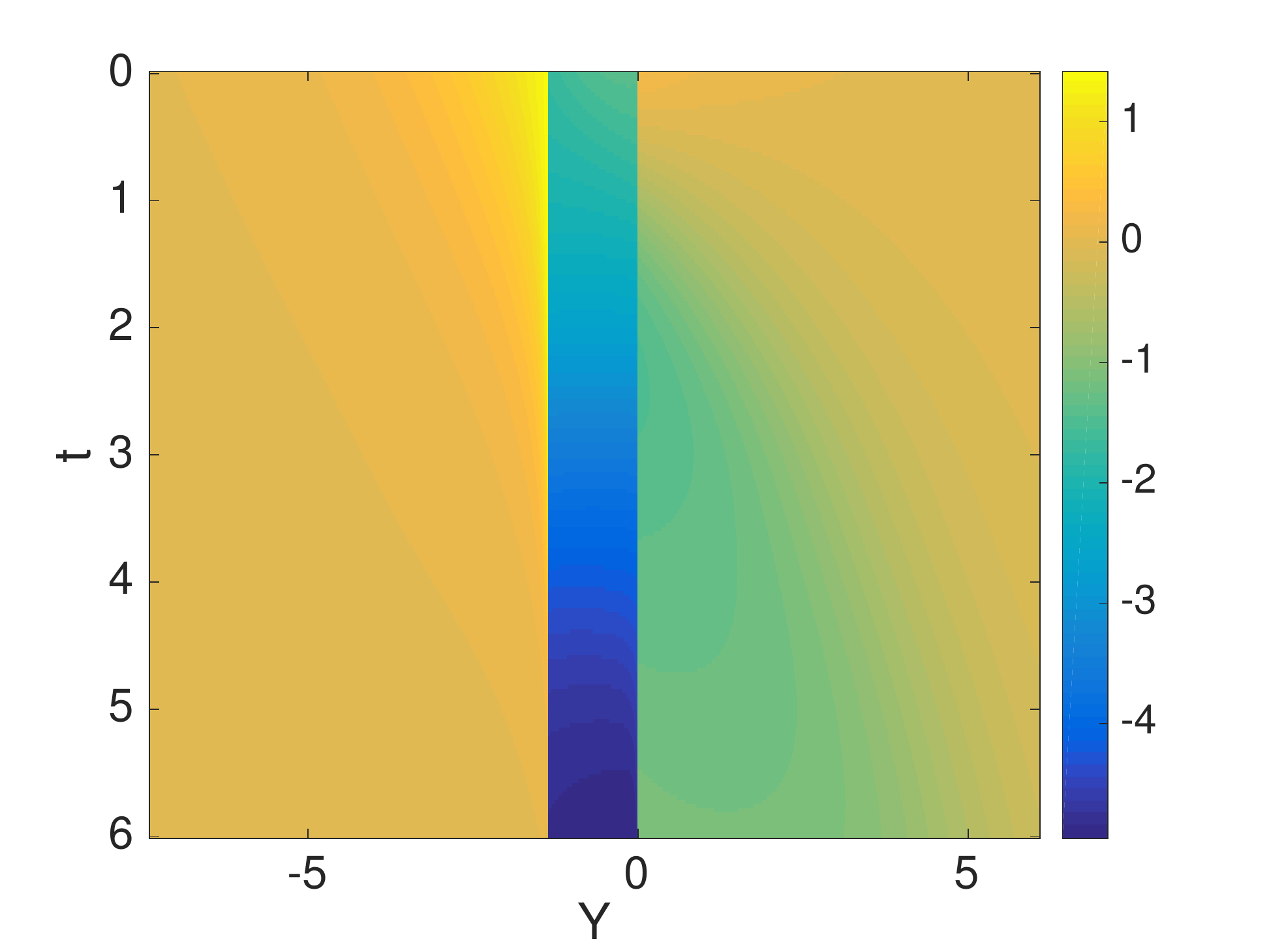}
  \caption{Function $\alpha(t,Y)$. When $\alpha=(2k+1)\pi$, with any integer $k$, $|u_x|$ blowup (when two peakons collide). In this example, $\alpha$ attains $-\pi$.}
  \label{appb_3}
\end{figure}

Then we study the energies $\mathcal{E}$ and $\mathcal{F}$, defined in \eqref{2.8}--\eqref{2.9}, between two characteristics $x=q_1(t)$ and $x=q_2(t)$ starting from the tips of two initial peakons, respectively.

We introduce several notations: First, we denote
\[
Y_1=Y(0, q_1(0)),\qquad Y_2=Y(0, q_2(0)),
\]
which are constant on two characteristics, respectively.
And we denote that $t=t_*$ is the time when two peakons first interact and 
\[
q_*:=q_1(t_*)=q_2(t_*).
\]

By \eqref{aluxi}, \eqref{2.8} and \eqref{2.9}, before the interaction of two peakons, we have
\begin{align*}
\mathcal{F}[q_1(t), q_2(t)]=&\int_{[q_1(t), q_2(t)]}  (u^4+2u^2u_x^2-\frac{1}{3}u_x^4)(t,x)\,dx\\
=&\int_{[Y_1, Y_2]} \LC u^4 \cos^4{\alpha\over2} + 2u^2 \cos^2{\alpha\over2} \sin^2{v\over2} - {1\over3} \sin^4{\alpha\over2} \RC \xi \ dY
\end{align*}
and
\[
\mathcal{L}[q_1(t), q_2(t)]:=\int_{[p_1(t), p_2(t)]}  u_x^4(t,x)\,dx=\int_{[Y_1, Y_2]}  (\sin^4{\alpha\over2})\, \xi \ dY.
\]

From Figure \ref{appb_4}, we see that $\mathcal{L}[q_1(t), q_2(t)]$ is always positive.
And when $t=t_*$,
\[
\mathcal{F}[ \{q_*\}]=\nu_{t_*}(\{q_*\})
=\int_{[Y_1, Y_2]} \LC u^4 \cos^4{\alpha\over2} + 2u^2 \cos^2{\alpha\over2} \sin^2{v\over2} - {1\over3} \sin^4{\alpha\over2} \RC \xi (t_*,Y)\ dY
\]
and
\[
\mathcal{L}[ \{q_*\}]=\int_{[Y_1, Y_2]}  (\sin^4{\alpha\over2}) \,\xi\, (t_*,Y)\ dY
\]
are also nonzero. This means that there exists fourth-order-energy concentration at the point $q_*$ when peakons interact in our example. This exactly shows the failure of natural Sobolev norm from energy law in studying the stability or Lipschitz continuous dependence of solutions, as shown in Figure \ref{concentration figure}.

Here, we note that by studying the solution on the $(t,Y)$-coordinates, we can avoid the difficulties caused by the blowup of $|u_x|$. One can clearly see that above integrals in the interval $[Y_1, Y_2]$ are ordinary integrals instead of improper ones.

\begin{figure}[tb]
 \includegraphics[scale=0.5]{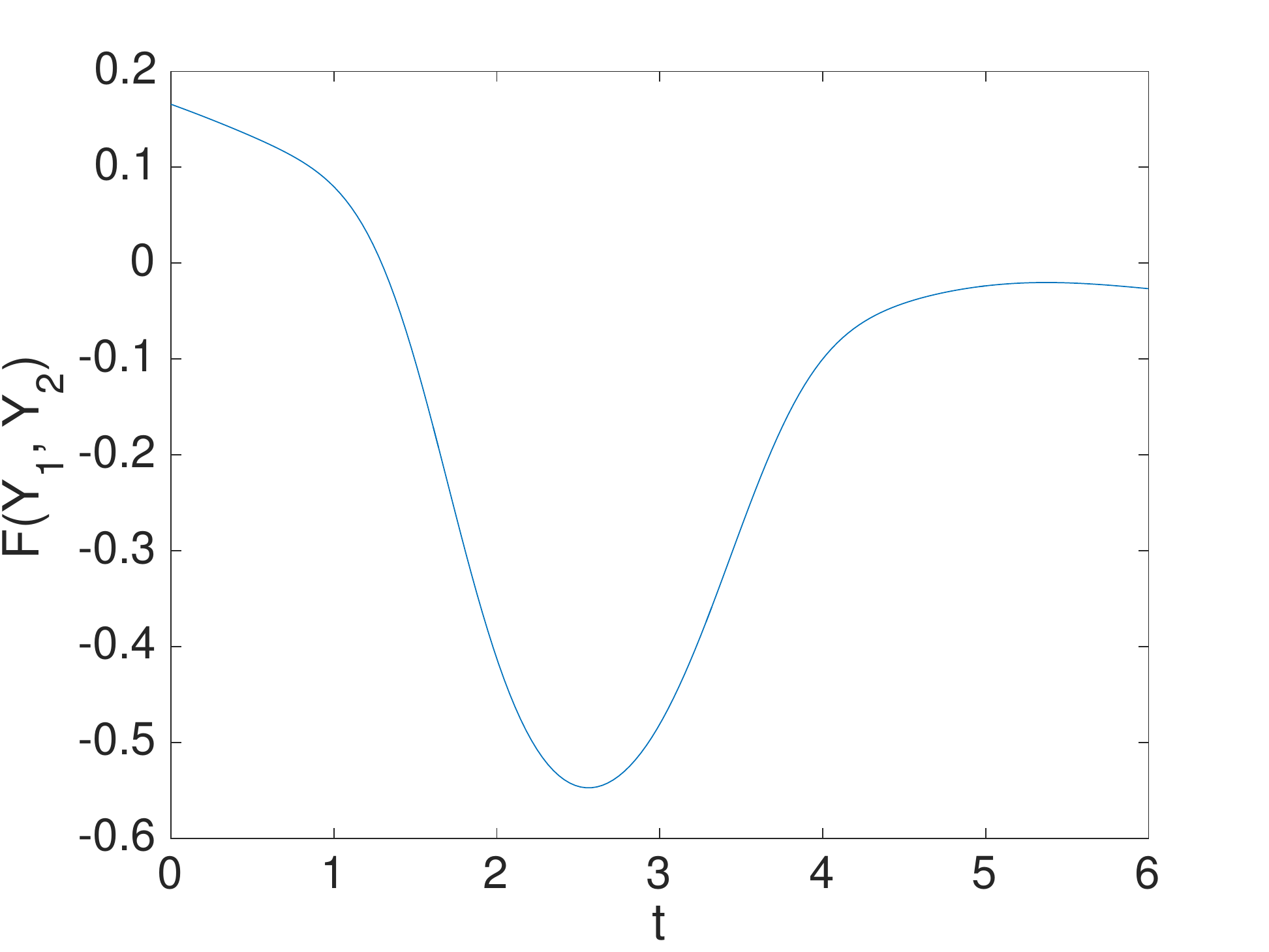}
  \includegraphics[scale=0.5]{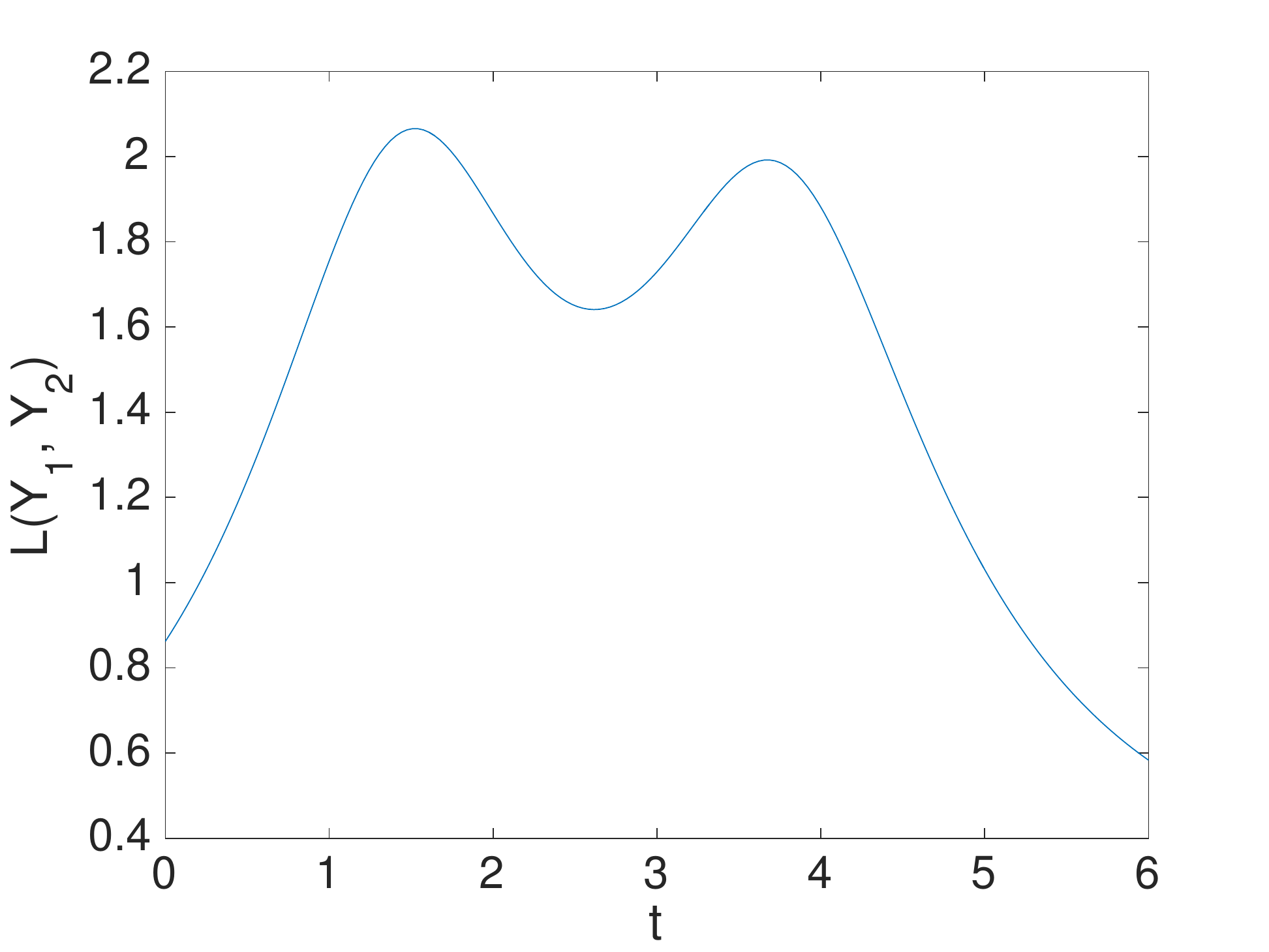}
  \caption{The functions $\mathcal{F}[q_1(t), q_2(t)]$ and $\mathcal{L}[q_1(t), q_2(t)]$.}
  \label{appb_4}
\end{figure}

Similarly, before the interaction of two peakons,
\[
\mathcal{E}[q_1(t), q_2(t)]=\int_{[q_1(t), q_2(t)]}  (u^2+u_x^2)(t,x)\,dx
=\int_{[Y_1, Y_2]} \LC u^2\cos^2{\alpha\over2} + \sin^2{\alpha\over2} \RC \xi \cos^2{\alpha\over2} \ dY.
\]
However, it is clear that when two peakons interact at time $t=t_*$, $\cos{\alpha\over2}=0$, hence,
\[
    \mathcal{E}\{q_*\}=
    \int_{[Y_1, Y_2]} \LC u^2\cos^2{\alpha\over2} + \sin^2{\alpha\over2} \RC \xi \cos^2{\alpha\over2} (t_*,Y)\ dY=0.
\]
This tells that there exists no second-order-energy  ($\mathcal{E}$) concentration for any weak solutions.

One can also find a picture of $u(t,x)$ on the original $(t,x)$-coordinates in Figure \ref{fig_uxt}.

\begin{figure}[tb]
 \includegraphics[scale=0.5]{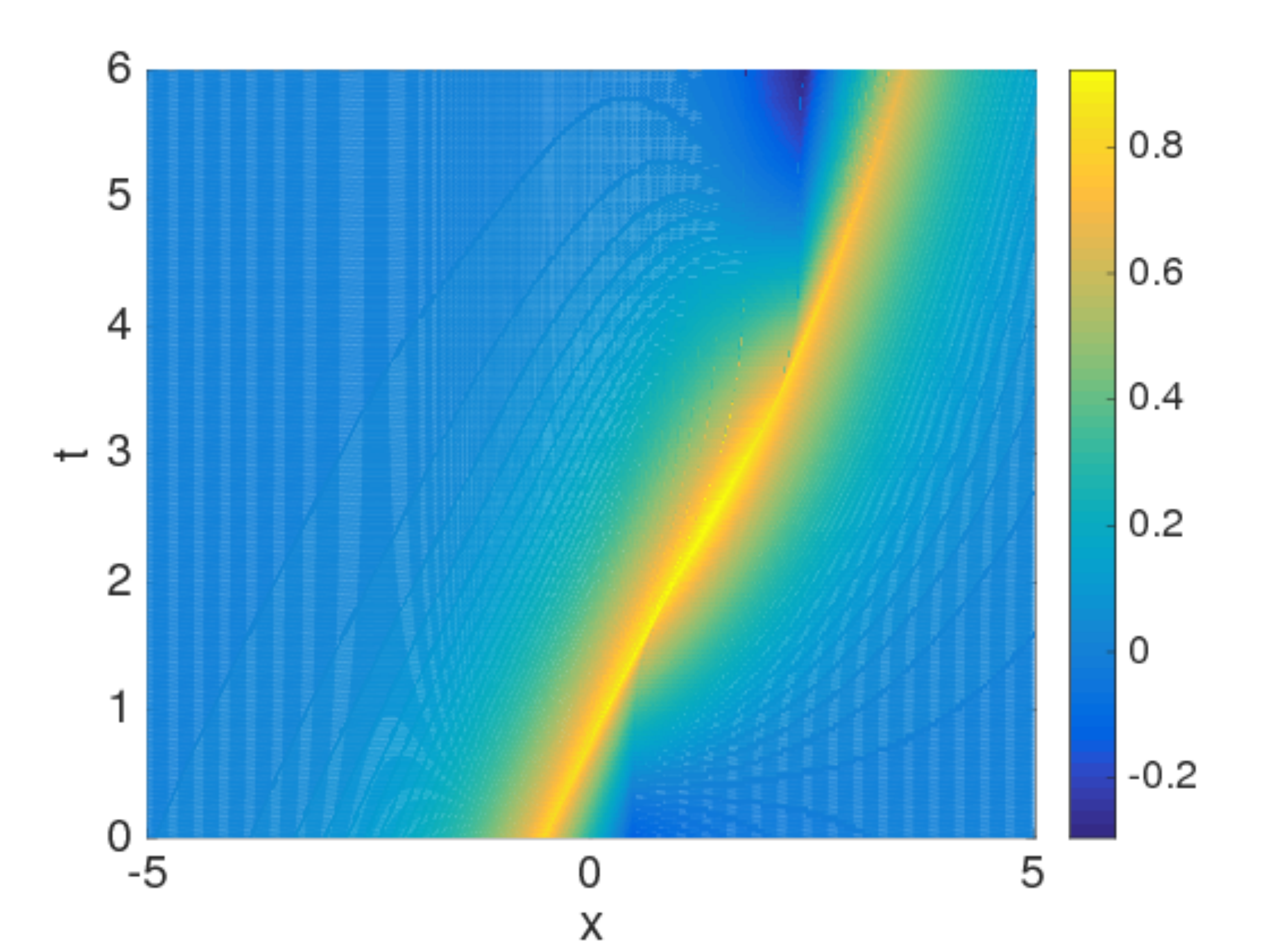}
  \caption{Solution $u(t,x)$.
  \label{fig_uxt}}
\end{figure}

\medskip

\noindent {\bf Acknowledgements.} The work of HC is partially supported by the National Natural Science Foundation of China-NSAF (No. 11271305, 11531010) and the China Scholarship Council No. 201506310110 as an exchange graduate student at Georgia Institute of Technology.
The work of RMC is partially supported by the Simons Foundation under Grant 354996.


\end{document}